\documentclass[a4paper,11pt,reqno,noindent]{amsart}
\usepackage[centertags]{amsmath}
\usepackage{amsfonts,amssymb,amsthm} 
\usepackage{hyperref}

 \hypersetup{
     pdfmenubar=false,        
     pdfnewwindow=true,      
     colorlinks=false,       
     linkcolor=blue,          
     citecolor=blue,        
     filecolor=magenta,      
     urlcolor=cyan           
 }
\usepackage{graphicx} 

\usepackage[english]{babel}
\usepackage{newlfont}
\usepackage{color}
\usepackage[body={15cm,21.5cm},centering]{geometry} 
\usepackage{fancyhdr}
\usepackage{esint}
\usepackage{enumerate}

\usepackage[active]{srcltx}

\newtheorem{theorem}{Theorem}
\newtheorem{lemma}{Lemma}
\newtheorem{proposition}{Proposition}

\newtheorem{remark}{Remark}
\newtheorem{definition}{Definition}

\theoremstyle{definition}

\newcommand{\sym}{\mathrm{sym}}

\newcommand{\dist}{\operatorname{dist}}
\newcommand{\supp}{\operatorname{supp}}

\newcommand{\e}{\varepsilon}

\newcommand{\Ihd}{I_{h,\Delta}}
\newcommand{\gd}{g_{\Delta}}
\newcommand{\ydel}{y^\Delta}
\newcommand{\Ehd}{E_{h,\Delta}}
\newcommand{\kvk}{\kappa^{\mathrm{vK}}}
\newcommand{\Ohd}{B_{1,\Delta}\times[-h/2,h/2]}
\newcommand{\ohd}{\omega_{h,\Delta}}
\newcommand{\Ahd}{\mathcal{A}_{h,\Delta}}
\newcommand{\Jhd}{J_{h,\Delta}}
\newcommand{\Jhdt}{\tilde J_{h,\Delta}}
\newcommand{\Q}{\mathcal Q}
\newcommand{\R}{\mathbb{R}}

\newcommand{\Z}{\mathbb{Z}}
\newcommand{\N}{\mathbb{N}}

\renewcommand{\d}{\mathrm{d}}
\renewcommand{\L}{\mathcal{L}}
\newcommand{\A}{\mathcal{A}}

\newcommand{\id}{\mathrm{Id}}

\newcommand\sgn{\mathrm{sgn}}

\renewcommand{\H}{\mathcal{H}}

\newcommand{\ve}{\epsilon}

\newcommand{\vk}{I^{\mathrm{vK}}}

\newcommand{\FvK}{F\"oppl-von K\'arm\'an }

\title{Energy scaling law for a single disclination in a thin elastic sheet}
\date{\today}
\author[H. Olbermann] {Heiner Olbermann}
\date{\today}
\address[Heiner Olbermann]{Hausdorff Center for Mathematics, Bonn, Germany}
\email{heiner.olbermann@hcm.uni-bonn.de}

\begin{document}

\maketitle
\begin{abstract}
We consider a single disclination in a thin elastic sheet of thickness $h$. We prove  ansatz-free
lower bounds for the free elastic energy in three different settings: First, for
a geometrically fully non-linear plate model, second, for three-dimensional
nonlinear elasticity, and third, for the \FvK plate theory. The
lower bounds in the first and third result are optimal in the sense that we
find upper bounds that are  identical to the respective lower bounds in the
leading order of $h$.
\end{abstract}

\section{Introduction}
\subsection{Setup and previous work}
Consider the following setup: Take a thin elastic sheet  in
the shape of a disc, and remove a sector from it. Then  glue the edges of the
cut back together. The (stress free) reference metric of the sheet is now that of
a singular cone. In other words, the reference metric is flat away from the
centre, where the Gauss curvature has a $\delta$-type singularity. In the
literature, such a  defect of the reference metric at the origin  
is often called a  ``disclination''. We are
interested in upper and lower bounds for the free elastic energy of the sheet
with a single disclination.
Note that we we do
not impose  boundary conditions or other constraints.\\
\\
 Disclinations play an
important role, e.g., in the modeling of cell walls
\cite{PhysRevA.38.1005,lidmar2003virus}, where they can serve as the vertices in
topologically spherical shells, that take on the shape of a (smoothed) icosahedron. Such a
behaviour has been observed, e.g., for virus capsids
\cite{caspar1962physical}. Recent interest in disclinations in thin sheets comes from   the
analysis of carbon
nanofilms \cite{MR2033874} and the modeling of biological growth processes \cite{MR2500622}.
An extensive review of the role of disclinations in
the mechanics of solids, not limited to the case of thin sheets, can be found in
\cite{MR2033874}.\\
\\
Considering a single  disclination of a comparable deficit
without boundary conditions, one expects that
 energy will be minimized by a shape that looks like a cone away
from the disclination, and deviates from the cone on a ball around the
disclination whose size is comparable to the thickness of the sheet.
In the literature, energy estimates for a
single disclination 
have so far been ansatz based (as in \cite{PhysRevA.38.1005,lidmar2003virus,MR3179439}),
or it has been assumed that the deformation is radially symmetric. In the latter
case, the dimensionality of the problem is reduced from two to one, and ODE
methods can be used. It has been shown in \cite{MR3148079} that under this
assumption, in the \FvK approximation,   minimizers do indeed show the behavior
described above. Additionally, this paper contains a
quantitative estimate on how fast  minimizers approach the singular cone as
the distance $r$ from the apex of the cone increases.\\
\\
In the recent work \cite{HOregular}, we suggested an ansatz how to prove lower
bounds for the free elastic energy of a single disclination without the
assumption of radial symmetry. We viewed the elastic sheet in the
deformed configuration as an immersed Riemannian manifold, and focused our
analysis on intrinsically defined objects, such as the metric and Gauss
curvature of the manifold. 
The main idea brought forward in this work was that estimates for
integrals of Gauss curvature can be obtained by interpolation between the metric
and the Gauss curvature. Lower bounds for integrals of Gauss curvature can be
translated to lower bounds on the bending energy by the isoperimetric inequality
on the sphere.
However, in \cite{HOregular}, we had to modify the elastic energy and make
simplifying assumptions on the shape of the sheet  to put this ansatz to
work.\\ 
\\
Here, we report on an improvement of the method of proof, that makes the
simplifying assumptions unnecessary altogether.  One further benefit of our
modified approach is that it makes
the use of the jargon of Riemannian geometry superfluous.

\subsection{Statement of results}
\label{sec:statement-results}
Our first result regards a single disclination in a geometrically fully
nonlinear plate model.
Let $B_1:=\{x\in\R^2:|x|<1\}$ be the sheet in the reference configuration. 
The singular cone may be described by the mapping $\ydel:B_1\to\R^3$,
\[
  \ydel(x)=\sqrt{1-\Delta^2} x+\Delta |x|e_3\,.
\]
Here, $0<\Delta<1$ is the height of the singular cone, and is determined by the deficit of the disclination at the
origin. The reference metric on $B_1$ is given by
\[
\begin{split}
  \gd(x)= & D\ydel(x)^TD\ydel(x)\\
=& (1-\Delta^2)\hat x\otimes \hat x+\Delta^2\hat x\otimes \hat x\\
=&\id_{2\times 2}-\Delta^2 \hat x^\bot\otimes \hat x^\bot\,,
\end{split}
\]
where $\hat x=x/|x|$ and $\hat x^{\bot}=(-x_2,x_1)/|x|$. The induced metric of a deformation $y\in W^{2,2}(B_1;\R^3)$
is 
\[
g_y=Dy^TDy\,.
\]
The free elastic energy   $\Ihd:W^{2,2}(B_1;\R^3)\to \R$ is defined by
\begin{equation}
\Ihd(y)= \int_{B_1}\left(|g_y-\gd|^2+h^2|D^2y|^2\right)\d \L^2\,,\label{eq:26}
\end{equation}
where $\d \L^2 $ denotes 2-dimensional Lebesgue measure.
The first result of the present paper is
\begin{theorem}
\label{thm:main1}
There exists a constant $C=C(\Delta)$ with the following property:
\[
2\pi\Delta^2h^2\left(|\log h|-\frac32\log|\log h|-C\right)\leq \min_{y\in
  W^{2,2}(B_1;\R^3)}I_{h,\Delta}(y)\leq 2
\pi\Delta^2h^2\left(|\log h|+C\right)
\]
for all small enough $h$.
\end{theorem}

Before we come to the statement of the other theorems, let us  try to explain the improvement of our techniques that allows us to
drop the additional assumptions we had to make in \cite{HOregular}. It consists in
noting that 
the function $\sum_{i=1}^3\det D^2y_i$ is controlled by both the membrane and
the bending term, in different functions spaces. Indeed, it is not difficult to
see that $\sum_{i=1}^3\det D^2y_i$ is bounded in $L^1$ by the bending
energy. The control exerted by the membrane energy is slightly more subtle. We
recall the very weak form of the Hessian:
\begin{equation}
\det D^2 w=(w_{,1} w_{,2})_{,12}-\frac12
  (|w_{,1}|^2)_{,22}-\frac12(|w_{,2}|^2)_{,11} \quad \text{ for } w\in C^2\,.\label{eq:23}
  \end{equation}
Hence, 
\begin{equation}
\sum_{i=1}^3\det D^2 y_i=(y_{,1}\cdot y_{,2})_{,12}-\frac12
  (|y_{,1}|^2)_{,22}-\frac12(|y_{,2}|^2)_{,11} \quad \text{ for } y\in C^2(B_1;\R^3)\,,\label{eq:23}
  \end{equation}
and we see that $\sum_{i=1}^3\det D^2y_i$ is controlled by the membrane term in
$W^{-2,1}$. 
Thus, by interpolation, we can get control over integrals of $\sum_{i=1}^3\det D^2
y_i$ in $L^1$. By the Sobolev inequality in $\R^2$, lower bounds for these
integrals can be translated to lower bounds for $|D^2y|$ (see Section \ref{sec:sobol-type-ineq}),  and hence for the
bending energy. \\
\\
As a consequence of Theorem \ref{thm:main1}, we will be able to prove a
scaling law in three-dimensional elasticity.
Let $\arccos:[-1,1]\to [0,\pi]$ denote the inverse of $\cos:[0,\pi]\to[-1,1]$.
Define 
\[
B_{1,\Delta}:=\left\{x=(x_1,x_2)\in B_1\setminus\{0\}:0\leq \arccos \frac{ x_1}{
|x|}< \sqrt{1-\Delta^2}\pi\right\}\,.
\] 
Let $\R_-:=\{(x_1,0):x_1\leq 0\}$, and let $\varphi:\R^2\setminus\R_-\to \R$
be the angular coordinate satisfying $x=|x|(\cos\varphi(x),\sin\varphi(x))$ with
values in $(-\pi,\pi)$.
We define the map $\iota_\Delta:\R^2\setminus\R_-\to B_1$  by
\[
\iota_\Delta(x)=\left(|x|\cos\frac{\varphi(x)}{\sqrt{1-\Delta^2}},|x|\sin\frac{\varphi(x)}{\sqrt{1-\Delta^2}}\right)\,.
\]
Now let
$\iota_{\Delta}^{(3)}:B_{1,\Delta}\times[-h/2,h/2]\to\R^2\times [-h/2,h/2]$ be defined by 
\[
\iota_{\Delta}^{(3)}(x',x_3)=\left(\iota_\Delta(x'),x_3\right)\,.
\]
Note that $(\iota_{\Delta}^{(3)})^{-1}$ is defined almost everywhere on $B_1\times [-h/2,h/2]$.
We define the set of allowed deformations to be
\[
\begin{split}
  \Ahd:=\Big\{&Y\in W^{1,2}(B_{1,\Delta}\times[-h/2,h/2];\R^3)\\
  &\text{ such that }Y\circ (\iota_{\Delta}^{(3)})^{-1}\in
  W^{1,2}(B_{1}\times[-h/2,h/2] ;\R^3)\Big\}\,.
\end{split}
\]
Let $W\in C(\R^{3\times 3};[0,\infty))$ be the stored energy density  that is
$SO(3)$-invariant, i.e., 
\[
W(RF)=W(F)\quad\text{ for all }R\in SO(3)\,,
F\in \R^{3\times 3}\,,
\] 
and such that it vanishes on $SO(3)$ and has quadratic growth, i.e.,
\begin{equation}
W(F)\geq C\dist^2(F,SO(3))\quad \text{ for
  all }F\in\R^{3\times 3}\,,\label{eq:32}
\end{equation}
 and  
\begin{equation}
W(F)<C_2\dist^2(F,SO(3)) \quad \text{ for }F\in\mathcal F_0\,,\label{eq:31}
\end{equation}
where $\mathcal F_0$ is some neighborhood of $SO(3)$ in $\R^{3\times 3}$.\\
The energy functional is given by $\Ehd:\Ahd\to\R$,

\begin{equation}
\Ehd(Y)=\frac{1}{h}\int_{\Ohd} W(DY)\d x\,.\label{eq:29}
\end{equation}

We will prove

\begin{theorem}
\label{thm:3d}
There exists a constant $C=C(\Delta,W)>0$ such that
\[
\frac{1}{C}h^2|\log h| \leq \min_{y\in\Ahd} \Ehd(y)\leq C h^2|\log h|\,
\]
for all  small enough $h$.
\end{theorem}

The third theorem is  the analogous result to Theorem \ref{thm:main1} in the \FvK
setting. Here, the free elastic energy $\vk_{h,\Delta}$ of the  deflection $(u,v)\in
W^{1,2}(B_1;\R^2)\times W^{2,2}(B_1)$ is given by 

\begin{equation}
\vk_{h,\Delta}(u,v)=\int  \left| 2\,\sym Du + Dv\otimes
  Dv+\Delta^2\,\hat x^\bot\otimes \hat x^\bot\right|^2+h^2|D^2v|^2\d x\,.\label{eq:27}
\end{equation}

Again, the parameter $\Delta$ determines the size of the deficit of the
disclination. For a heuristic derivation of $\vk_{h,\Delta}$ from $\Ihd$ see
Section \ref{sec:fvk-setting-proof}.
Our third result is
\begin{theorem}
\label{thm:main}
There exists a constant $C=C(\Delta)>0$ with the following property: 
\begin{equation}
2\pi\Delta^2h^2\left(|\log h|-2\log|\log h|-C\right) \leq   \min_{(u,v)\in\A }\vk_{h,\Delta}(u,v)\leq 2\pi\Delta^2 h^2(|\log h|+C)\label{eq:3}\,
 \end{equation}
for all small enough $h$.
\end{theorem}

\subsection{Scientific context}
In the last two and a half decades or so, there has been  a lot of interest in
the energy focusing in thin elastic sheets, predominantly in the physics and
engineering community. The structures under consideration are typically sharp
folds, conical singularities, or a (possibly very complex) combination of both.
The first analysis of the scaling law for the free elastic
energy in the sense of an asymptotic expansion for a thin elastic shell
displaying sharp folds can be found in 
\cite{witten1993asymptotic}.  The so-called ``minimal ridge''
has been analyzed in \cite{1996PhRvE..53.3750L}, see also
\cite{PhysRevLett.87.206105,1997PhRvE..55.1577L,PhysRevLett.78.1303,Lobkovsky01121995}. Structures
with stress focusing in vertices (the so-called \emph{d-cones})
were first investigated in \cite{MR1447150}, see also
\cite{CCMM,PhysRevLett.80.2358,Cerda08032005}. To the interested reader, we
recommend the overview article by Witten \cite{RevModPhys.79.643}, where many
more references can be found.\\
\\
In the mathematical literature, the most closely related articles to the present
work  are
\cite{MR2358334,MR3102597,MR3168627}. In \cite{MR2358334},  it has been proved that  the elastic energy per unit thickness of a ``single fold''  scales with $h^{5/3}$, building on results from \cite{MR2023444}. In
\cite{MR3168627,MR3102597}, the following  has been proved: Take an 
elastic sheet in the shape of a disc, and prescribe the deformation of 
the sheet at the center and at the boundary, such that it coincides with a
(singular)  
conical deformation there. In this class of deformations, the
elastic energy per  unit thickness scales with $h^2|\log h|$.\\
On a technical level, the works \cite{MR2358334,MR3102597,MR3168627} treat
variational problems for the  functional \eqref{eq:26}, where the
reference metric $\gd$ is replaced by the flat 2-dimensional Euclidean metric,
with certain Dirichlet boundary conditions. These boundary conditions 
 are of a very special type: They are chosen such that there exists a unique
 Lipschitz continuous isometric immersion\footnote{By a 
 Lipschitz continuous isometric immersion, we mean a Lipschitz map $\bar
 y:B_1\to \R^3$ with the property that $D\bar y^TD \bar y$ equals the reference
 metric almost everywhere.} that satisfies them, which possesses infinite
bending energy. In the sequel, we
 will call this Lipschitz
 continuous isometric immersion the ``singular reference
 configuration''. Also, we will call boundary conditions ``tensile'', if a
 singular reference configuration exists.
The idea of proof in  \cite{MR2358334,MR3102597,MR3168627} is that deviations from the singular reference configuration are penalized by the membrane energy.
The optimal balance between membrane and bending energy yields the lower bound
in the elastic energy.\\
\\
In the  setting of the single disclination, there are no tensile boundary conditions -- there are
none at all. \\
\\
To the best of our knowledge, the method of proof using tensile boundary conditions has so far
been the only one for proving ansatz-free lower bounds in variational models for thin
elastic sheets with energy focusing in folds and vertices (not counting the
predecessor \cite{HOregular} of the current work). We expect that the method of
proof we present here will have applications in  variational models that are
sufficiently closely related to the single disclination. An interesting open
question is whether or not our
techniques can also be applied  in  settings with a flat reference metric.
\subsection{$h$-principle and rigidity}
We  want to put Theorem \ref{thm:main1}  in  yet another mathematical context. Recall the
results by Nash \cite{MR0065993} and Kuiper \cite{MR0075640}, stating that any
strictly short immersion $y:M\to\R^{n+1}$ of an $n$-dimensional  Riemannian
manifold $(M,g)$ can be arbitrarily well approximated by a $C^1$ isometric
immersion $\tilde y$ in the $C^0$ norm. I.e., given $\e>0$ and an  immersion
$y\in C^1(M;\R^{n+1})$ with $D y^TD y<g$ everywhere, there
exists $\tilde y\in C^1(M;\R^{n+1})$ with $D\tilde y^TD\tilde y=g$ and
$\|y-\tilde y\|_{C^0}<\e$. These results are relevant for the present context,
since the leading term in the elastic energy \eqref{eq:26}, the membrane term, penalizes
deviations from isometric immersions. By the Nash-Kuiper Theorem, there exists a
huge amount  of mutually vastly different shapes that have arbitrarily small
membrane energy. In the present setting, one needs to show that all of these many degrees of freedom
carry large bending energy.\\
\\
Next, recall the well known Weyl problem from differential geometry: Given a
Riemannian metric $g$ with positive Gauss curvature on the two-sphere $S^2$,
find an isometric immersion $y:S^2\to \R^3$. Note that by the Nash-Kuiper result
above, there exist many solutions if we only require $y\in C^1$. On the other
hand, by a result by
Pogorelov \cite{MR0346714}, the solution to the Weyl problem is unique (up to
Euclidean motions)
in the class of immersions with \emph{bounded extrinsic curvature}. This class
is defined by requiring that the pull-back of the volume form on $S^2$  under the Gauss map defines a
signed Radon measure. For smooth
maps, the total variation of that measure is just the $L^1$-norm of the Gauss
curvature. We see that control over curvature ``eliminates'' the Nash-Kuiper constructions.\\
\\
On a heuristic level, our approach can be viewed as a quantitative version of Pogorelov's rigidity
result. In order to see how, note that
the different results can be formulated as a statement on the measure 
\[
k_y:U\mapsto \int_U K_y\d A_y\,,
\]
defined on open sets $U$,  where $K_y\d A_y$ is a certain curvature
form (in Pogorelov's work and \cite{HOregular}, the Gauss curvature form; for the
geometrically fully nonlinear model treated here, $\sum_{i=1}^3\det D^2 y_i\d x$;
for the \FvK case, $\det D^2v \,\d x$, which is the \FvK version of Gauss
curvature). The crucial part of Pogorelov's result is that the measure
$k_y$ is identical to the pull back of the volume form on $S^2$ under the Gauss map if
the latter is a Radon measure. Here, we give certain  quantitative estimates for the
difference $k_y-k_0$, where $k_0$ is the curvature measure
associated to the reference configuration. These estimates allow for
the construction of a lower bound of the bending energy via a suitable
Sobolev type inequality.\\
\\
By imposing tensile boundary conditions as in \cite{MR2358334,MR3102597,MR3168627}, one eliminates strictly short maps from
the set of allowed configurations. In this way, the problem described above is
circumvented. Here, we describe for the first time a setting where the
constructions in the style of Nash-Kuiper are allowed by the boundary data, but
shown to be energetically unfavorable. 
\subsection{Notation and plan of the paper} 
The canonical basis vectors of $\R^n$ will be denoted by $e_1,\dots,e_n$. For a vector $x\in \R^n$, we
denote its components by $x_1,\dots,x_n$. For a vector $v=(v_1,v_2)\in
\R^2$, let $v^\bot:=(-v_2,v_1)$. Open balls in
$\R^2$ are denoted by $B_r:=\{x\in\R^2:|x|<r\}$. 
The set of linear isometries from $\R^2$ to $\R^3$ is denoted by $O(2,3)$.\\
For a function $f$ on $\R^2$,
we use the notation $f_{,i}=\partial_{x_i}f$ for the partial derivatives.\\
The volume of the unit ball in
$\R^n$ will be denoted by $\omega_n:=\pi^{n/2}/\Gamma(n/2+1)$.
The
symbol $\L^n$ denotes $n$-dimensional Lebesgue measure, and $\H^k$ is the
$k$-dimensional Hausdorff measure. \\
The symbol $C$ denotes numerical
constants that only depend on $\Delta$. A statement such as ``$f\leq Cg$'' is to
be understood as ``there exists a constant $C=C(\Delta)>0$ such that $f\leq C
g$''. The value of $C$ may vary from one line to the next.
By $g(h)=o(f(h))$ and $\tilde g(h)=O(\tilde f(h))$, we mean that
\[
\limsup_{h\to 0} \frac{|g(h)|}{|f(h)|}=0\,,\quad \limsup_{h\to 0} \frac{|\tilde
  g(h)|}{|\tilde f(h)|}<\infty\,
\]
respectively.\\
\\
This paper is structured as follows:
In Section \ref{sec:sobol-type-ineq}, we state a Sobolev type inequality
for the Brouwer degree, which will help us to translate control over $k_y$ into
a lower bound for the bending energy later on. Theorems \ref{thm:main1},
\ref{thm:3d} and  \ref{thm:main} are proved in Sections
\ref{sec:proof-theor-refthm:m}, \ref{sec:three-dimens-elast} and
\ref{sec:fvk-setting-proof} respectively.

\section{A Sobolev type  inequality for the Brouwer degree}
\label{sec:sobol-type-ineq}
In the proof of Theorems \ref{thm:main1} and \ref{thm:main} below, we will use the Sobolev inequality to get a lower bound on the
bending term. From smallness of the membrane and bending energies, we will have
for most $r\in[0,1]$,
a lower bound on $\sum_{i=1}^3\H^2(Dy_i(B_r))$ in the geometrically fully
nonlinear model, and on  $\H^2(Dv(B_r))$ in the \FvK model (see Propositions \ref{prop:interpol1} and \ref{prop:interpolvk}). The Sobolev inequality will give us then a lower bound for
$|D^2 y|$ (in the first case) and $|D^2v|$ (in the second) on $\partial
B_r$.\\ 
In the proof, we are going to need some standard facts concerning the Brouwer
degree. For a more thorough exposition including proofs of the  claims made
here, see e.g.~\cite{MR1373430}.\\ 
Let  $U\subset\R^n$ be bounded, 
$f\in C^\infty(\overline{U};\R^n)$, and   
$y\in \R^n\setminus f(\partial U)$. Let $A_{y,f}$ denote the connected component
of $\R^n\setminus f(\partial U)$ that contains $y$, and let 
$\mu$ be a smooth $n$-form on $\R^n$
with support in $A_{y,f}$ and $\int_{\R^n}\mu=1$. Then the  Brouwer degree
$\deg(y,U,f)$ is defined by
\[
\deg(y,U,f)=\int_U f^*\mu\,,
\]
where $f^*$ denotes the pull-back under $f$. It can be shown that this definition
is independent of the choice of $\mu$, and that $\deg(\cdot,U,f)$ is
integer-valued, and constant on connected components of $\R^n\setminus
f(\partial U)$. By approximation with smooth functions, it can be shown that the
Brouwer degree has a well defined meaning also for $f\in C^0(\overline U;\R^n)$
and $y\in\R^n \setminus f(\partial U)$.\\
From now on, let us suppose that $U\subset\R^n$ is bounded and open with
Lipschitz boundary. For $f\in C^1(\overline U;\R^n)$,  and an $n$-form $\mu$ with coefficients in $L^\infty(\R^n)$, we have
\[
\int_{\R^n} \deg(y,U,f)\mu=\int_U f^*\mu\,.
\]
This follows immediately from the defining relation for the Brouwer degree by
approximation with smooth functions. If $\mu$ is an exact form, i.e., $\mu=\d
\omega$ for some $n-1$ form $\omega$ with coefficients in $W^{1,\infty}$, then
\begin{equation}
\begin{split}
  \int_{\R^n} \deg(\cdot,U,f)\mu=&\int_U f^*(\d\omega)\\
   =&\int_U \d\,(f^*\omega)\\
  =&\int_{\partial U}f^*\omega\,.
\end{split}\label{eq:11}
\end{equation}
\renewcommand{\div}{\mathrm{div}\,}
We also recall that the space $BV(\R^n)$ is defined as the set of $g\in
L^1(\R^n)$ such that the  total variation $|Dg|(\R^n)$  is finite,
\begin{equation}
|Dg|(\R^n)=\sup\left\{\int_{\R^n} g\,\div h\, \d x:h\in C^1(\R^n;\R^n),\,|h(x)|\leq 1 \text{ for all
  } x\in\R^n\right\}\,.\label{eq:10}
\end{equation}
Let $C^1(\R^n;\wedge^{n-1})$ denote the set of $n-1$ forms on
$\R^n$ with coefficients in $C^1(\R^n)$,
\[
\begin{split}
  C^1(\R^n;\wedge^{n-1})=&\Big\{\sum_{i_1<\dots<i_{n-1}}\omega_{i_1\dots
    i_{n-1}}
  \d x_{i_1}\wedge\dots\wedge \d x_{i_{n-1}}:\\
  &\omega_{i_1\dots i_{n-1}}\in C^1(\R^n) \text{ for all
  }i_1,\dots,i_{n-1}\in\{1,\dots,n\}\Big\}\,.
\end{split}
\]
Note that the space of $n-1$ forms at a given point $x$ is canonically
isomorphic to $\R^n$, via
\[
\d x_{\sigma(1)}\wedge\dots \wedge\d x_{\sigma(n-1)}\mapsto \sgn\, \sigma
\,e_{\sigma(n)}\,,
\]
for any permutation $\sigma:\{1,\dots,n\}\to\{1,\dots,n\}$.
 When we write $|\omega(x)|$, we mean the norm of the $n-1$
form $\omega(x)$ under that isomorphism.

By \eqref{eq:11} and \eqref{eq:10}, we have

\begin{equation}
\begin{split}
  |D(\deg(\cdot,U,f))|(\R^n)=&\sup\left\{\int_{\R^n} \deg(\cdot,U,f)\div h \d x:h\in
    C^1(\R^n;\R^n),\,|h|\leq 1 \right\}\\
  =& \sup\left\{\int_{\partial U} f^*\omega: \omega\in
    C^1(\R^n;\wedge^{n-1}),\,|\omega|\leq 1 \right\}\\
  = & \int_{\partial U} |Df^\parallel|\d \H^{n-1}\,,
\end{split}\label{eq:12}
\end{equation}
where $Df^\parallel(x)$ denotes the restriction of the gradient of $Df$ to the
 space tangent to $\partial U$ (which exists $\H^{n-1}$ almost everywhere by assumption). 
Finally, recall the Sobolev inequality for $BV$ functions on $\R^n$ (see
e.g.~\cite{MR2039959}, Theorem 2.1),
\begin{equation}
\left(\int_{\R^n} |f|^{n/(n-1)}\d x\right)^{(n-1)/n}\leq C(n)|Df|(\R^n)\,,\label{eq:25}
\end{equation}
where $C(n)=(n\omega_n^{1/n})^{-1}$ is  the isoperimetric constant
in $\R^n$.

\begin{lemma}
\label{lem:isoper}
For $v\in C^2(\overline{B_1})$ and $0\leq r\leq 1$,
\[
\int_{\partial B_r}|D^2 v|\d\H^1\geq \left(4\pi \left|\int_{B_r}\det
    D^2v\d x\right|\right)^{1/2}\,.
\]
\end{lemma}
\begin{proof}
  We write $w:=Dv:B_1\to \R^2$. By approximation with smooth functions, we may
  assume $w\in C^\infty(B_1;\R^2)$. By \eqref{eq:12}, we have

  \begin{equation}
\begin{split}
  |D(\deg(\cdot,B_r,w))|(\R^2)\leq&\int_{\partial B_r}|Dw|\d\H^1\,.
\end{split}\label{eq:16}
\end{equation}
By \eqref{eq:25}, we have
\begin{equation}
  \begin{split}
    |D(\deg(\cdot,B_r,w))|(\R^2)\geq& \left(4\pi\int_{\R^2}\left|\deg(y,B_r,w)\d y\right|^2\right)^{1/2}\\
\geq    & \left(4\pi\left|\int_{\R^2}\deg(y,B_r,w)\d y\right|\right)^{1/2}\\
    = & \left(4\pi\left|\int_{B_r}\det D w\d x\right|\right)^{1/2}\,,
  \end{split}\label{eq:24}
\end{equation}
where we used the fact $\deg(x,B_r,w)\in\Z$, and hence $|\deg(x,B_r,w)|^2\geq |\deg(x,B_r,w)|$ for almost every $x$, in the second inequality.
Combining \eqref{eq:16} and \eqref{eq:24} yields the claim of the lemma.
\end{proof}

\section{Proof of Theorem \ref{thm:main1}}
\label{sec:proof-theor-refthm:m}
\begin{lemma}
We have
\label{lem:upperbound1}
  \begin{equation*}
  \inf_{y\in W^{2,2}(B_1;\R^3)} \Ihd(y)\leq 2\pi\Delta^2h^2\left(|\log h|+C\right)\,.
\end{equation*}
\end{lemma}
\begin{proof}
Recall the definition of $\ydel$,
\[
\ydel(x)=\sqrt{1-\Delta^2} x+  \Delta |x|e_3\,.
\]
 Let  $\eta\in C^\infty([0,\infty))$ with $\eta(t)=0$ for $t\leq \frac12$, $\eta(t)=1$ for $t\geq
  1$ and  $|\eta'|\leq 4$, $|\eta''|\leq 8$. 
We set 
\[
y_0(x)=\eta(|x|/h) \ydel(x)\,.
\]
We compute
\[
\begin{split}
  D y_0(x)=&\frac{\eta'(x/h)}{h}\ydel\otimes \hat
  x+\eta(x/h)\sqrt{1-\Delta^2}\left(\begin{array}{ccc}1 & 0 &
      0\\0&1&0\end{array}\right)^T+\Delta e_3\otimes \hat x\,,\\
  D^2 y_0(x)=&\frac{\eta''(x/h)}{h^2}\ydel\otimes \hat x\otimes\hat x
  + \frac{\eta'(x/h)}{h}\Bigg[\sqrt{1-\Delta^2}(2\hat x\otimes \hat x\otimes \hat
    x+\hat x^\bot\otimes \hat x\otimes \hat x^\bot\\
&+\hat x^\bot \otimes \hat
    x^\bot\otimes \hat x)
+2\Delta e_3\otimes \hat x\otimes \hat x+\frac{1}{|x|}\ydel\otimes
   \hat x^\bot\otimes\hat x^\bot\Bigg]\\
 & +\frac{\eta(x/h)}{|x|}\Delta e_3\otimes \hat x^\bot\otimes \hat x^\bot\,.
\end{split}
\]
On $B_1\setminus B_h$, we have $y_0=\ydel$, $\eta'(|\cdot|/h)=0$, and hence $Dy_0^TDy_0=(D\ydel)^TD\ydel$ and $D^2y_0= |x|^{-1}\Delta e_3\otimes \hat x^\bot\otimes \hat x^\bot$.
On $B_h$, we see from the formulas above that $|Dy_0|\leq C$ and $|D^2y_0|\leq
Ch^{-1}$. (To see the latter, one uses $|\ydel|\leq h$ on $B_h$.)
Hence the energy can be estimated by
\[
\begin{split}
  I_h(y_0)\leq &\int_{B_h} \left(C+h^2 (Ch^{-1})^2)\right)\d x
  +h^2\int_{B_1\setminus B_h}|x|^{-2}\Delta^2\d x\\
  \leq &2\pi\Delta^2h^2(|\log h|+C)\,.
\end{split}
\]
This proves the lemma.
\end{proof}

For the following definition, recall the formulation of the Hessian in its very
weak form, equation \eqref{eq:23}.

\begin{definition}
 Let $i\in\{1,2,3\}$.  By $\det D^2\ydel_i$, we denote the distribution
\[
\varphi\mapsto-\frac12\int_{B_1}\d\varphi\wedge\left(\ydel_{i,1}\d \hat
  y_{i,2}-\ydel_{i,1}\d \ydel_{i,2}\right)\,.
\]
\end{definition}
\begin{lemma}
\label{lem:hatydist}
  The distributions $\det D^2\ydel_i$, $i=1,2,3$, can be extended to  Radon
  measures. Denoting the latter  with the same symbol, we have
\[
\begin{split}
  \det D^2\ydel_i=&0\quad\text{ for } i=1,2\\
  \det D^2\ydel_3=&\Delta^2\pi\delta_0\,,
\end{split}
\]
where $\delta_0$ is the Dirac distribution $\delta_0:\varphi\mapsto \varphi(0)$.
Furthermore, we have
  
\begin{equation}
  \det D^2 \ydel_i=(\ydel_{i,1}\cdot \ydel_{i,2})_{,12}-\frac12 (|\ydel_{i,1}|^2)_{,22}-\frac12(|\ydel_{i,2}|^2)_{,11}\label{eq:1}
  \end{equation}
  in the sense of distributions.
\end{lemma}
\begin{proof}
For $i=1,2$, the form $\ydel_{i,1}\d\ydel_{i,2}-\ydel_{i,1}\d \ydel_{i,2}$ vanishes almost
  everywhere. Furthermore, we compute
\begin{equation}
\begin{split}
  \ydel_{3,1}\d \hat
  y^3_{,2}-\ydel_{3,1}\d \hat
  y^3_{,2}=&\Delta^2\Bigg[
\frac{x_1}{|x|}\left(-\frac{x_1x_2}{|x|^3}\d
    x_1+\frac{x_1^2}{|x|^3}\d x_2\right)\\
&-\frac{x_2}{|x|}\left(-\frac{x_1x_2}{|x|^3}\d
    x_2+\frac{x_2^2}{|x|^3}\d x_1\right)\Bigg]\\
=&\Delta^2\left[ \frac{x_1}{|x|^2}\d x_2-\frac{x_2}{|x|^2}\d x_1\right]\,.
\end{split}\label{eq:2}
\end{equation}
It is well known that the differential of  the right hand side is just
$-\Delta^22\pi\delta_0 \d x_1\wedge \d x_2$.
This proves the first claim of the lemma. The  claim \eqref{eq:1}  follows
from a simple integration by parts,
\[
\begin{split}
 \Bigg<(&\ydel_{i,1} \ydel_{i,2})_{,12}-\frac12 (|\ydel_{i,1}|^2)_{,22}-\frac12(|\ydel_{i,2}|^2)_{,11},\,\varphi\Bigg>\\
  &= \frac12 \int_{B_1}\left(\varphi_{,1}\left((\ydel_{i,1} \ydel_{i,2})_{,2}-(|\ydel_{i,2}|^2)_{,1}\right)+\varphi_{,2}\left( (\ydel_{i,1}
  \ydel_{i,2})_{,1}- (|\ydel_{i,1}|^2)_{,2}\right)\right)\d x\\
&=\left<\det D^2 \ydel_i,\varphi\right>\,.
\end{split}
\]
This completes the proof of the lemma.
\end{proof}
For a given deformation $y\in C^2(B_1;\R^3)$,  let $\kappa_y$ be defined by

\begin{equation}
\begin{split}
  \kappa_y(\rho)=&\left<\sum_{i=1}^3\det D^2 y_i,\chi_{B_{\rho}}\right>\\
=&\int_{B_\rho}\sum_{i=1}^3 \det D^2 y_i\,\d x\\
=& \sum_{i=1}^3 \frac12\int_{\partial B_\rho} y_{i,1}\d y_{i,2}-y_{i,1}\d y_{i,2} \,.
\end{split}\label{eq:14}
\end{equation}

\begin{proposition}
\label{prop:interpol1}
Let $y\in C^2(B_1;\R^3)$ with 
\[
\int_{B_{1-5h}\setminus B_{5h}}|g_y-g_\Delta|^2+h^2|D^2y|^2\d\L^2\leq C h^2|\log h|\,.
\]
Then 
\[
\|\kappa_y-\pi\Delta^2\|_{L^1(6h,R)}\leq C h^{1/2}R^{1/2}|\log h|^{3/4}\quad
\text{ for all }R\in [2h,1-5h]\,.
\]
\end{proposition}
\begin{remark}
We could prove Theorem \ref{thm:main1} via  Propositions \ref{prop:interpol1}
and \ref{prop:diadicannuli} also if we made the stronger assumption $\Ihd(y)\leq
C_1h^2|\log h|$ in Proposition \ref{prop:interpol1}. However, the weaker
assumption will be convenient in the proof of the result on  three-dimensional
elasticity, Theorem \ref{thm:3d}.
\end{remark}
\begin{proof}
Let
$h_0=h_0(y)\in[5h,6h]$ be chosen such that
\[
\int_{\partial B_{h_0}}|g_y-\gd|^2\d\H^1=\int_{\partial B_{h_0}}|Dy^TDy-(D\ydel)^TD\ydel|^2\d \H^1\leq C h|\log h|\,.
\]
For $r\in[h_0,1]$, we set
\[
\begin{split}
  F_1(r)=&\int_{h_0}^r\kappa_y(\rho)\d\rho\,,\\
  F_2(r)=&(r-{h_0})\pi\Delta^2\,,\\
F=&F_1-F_2\,.
\end{split}
\]
By these definitions,

\[
  F(r)
=\int_{h_0}^r \sum_{i=1}^3\left<\det D^2y_i-\det D^2 \ydel_i,\chi_{B_\rho}\right>\d\rho\,.
\]
Now we 
 set
\[
u_{ij}=(Dy^TDy-(D\ydel)^TD\ydel)_{ij}\,.
\]

By the assumption $y\in C^2(B_1;\R^3)$ and Lemma \ref{lem:hatydist}, we have
\[
F(r)=\int_{h_0}^r\d\rho\int_{B_\rho}\left[u_{12,12}-\frac12 u_{11,22}-\frac12 u_{22,11}\right]\d
  \L^2\,.
\]

Let $k,l\in\{1,2\}$, and 
$f\in \{u_{11},u_{12},u_{22}\}$. 
We repeatedly apply Gauss' Theorem,

\begin{equation}
\begin{split}
  \int_{h_0}^r\d\rho\int_{B_\rho}f_{,kl}(x)\d x=&
  \int_{h_0}^r\d\rho\int_{\partial B_\rho}\frac{x_l}{|x|}f_{,k}(x)\d\H^1(x)
  \\
  =& \int_{B_r\setminus B_{h_0}} \frac{x_l}{|x|}f_{,k}(x)\d
  x\\
  =& \int_{\partial (B_r\setminus B_{h_0})}
  \frac{x_k}{|x|}\frac{x_l}{|x|}f(x)\d\H^1(x)-\int_{B_r\setminus
    B_{h_0}} \left(\frac{x_l}{|x|}\right)_{,k}f(x)\d x
\end{split}\label{eq:44}
\end{equation}
We will estimate the $L^1$ norm of 
the function $G:r\mapsto \int_{h_0}^r\d\rho\int_{B_\rho}f_{,kl}(x)\d x$. By the
triangle inequality and \eqref{eq:44}, we have
\[
\begin{split}
  \|G\|_{L^1(h_0,R)}\leq &\int_{h_0}^R\left|\int_{\partial B_r}
    \frac{x_k}{|x|}\frac{x_l}{|x|}f(x)\d\H^1(x)\right|\d r\\
  &+\int_{h_0}^R\left|\int_{\partial B_{h_0}}
    \frac{x_k}{|x|}\frac{x_l}{|x|}f(x)\d\H^1(x)\right|\d r\\
  &+\int_{h_0}^R\left|\int_{B_r\setminus B_{h_0}}
    \left(\frac{x_l}{|x|}\right)_{,k}f(x)\d x\right|\d r
\end{split}
\]
We estimate the terms appearing on the right hand side
using the Cauchy-Schwarz inequality. For the first one, we have
\[
\begin{split}
  \int_{h_0}^R \left|\int_{\partial B_r}
    \frac{x_k}{|x|}\frac{x_l}{|x|}f(x)\d\H^1(x)\right|\d r&\leq
  \int_{B_R\setminus B_{h_0}}\left|\frac{x_k}{|x|}\frac{x_l}{|x|}f(x)\right|\d\L^2(x)\\
  \leq & \left(\int_{B_R\setminus B_{h_0}}|f|^2\d\L^2\right)^{1/2}(\pi
  (R^2-h_0^2)^{1/2}
\end{split}
\]
Secondly, we have
\[
\begin{split}
  \int_{h_0}^R \left|\int_{\partial B_{h_0}}
    \frac{x_k}{|x|}\frac{x_l}{|x|}f(x)\d\H^1(x)\right|\d r&= 
  (R-h_0)\left|\int_{\partial B_{h_0}}
    \frac{x_k}{|x|}\frac{x_l}{|x|}f(x)\d\H^1(x)\right|\\
  &\leq (R-h_0) \left(\int_{\partial B_{h_0}}
    |f|^2\d\H^1\right)^{1/2}\left(2\pi h_0\right)^{1/2}
\end{split}
\]
For the third term,
\[
\begin{split}
  \int_{h_0}^R \left|\int_{B_r\setminus B_{h_0}}
    \left(\frac{x_l}{|x|}\right)_{,k}f(x)\d x\right|\d r\leq &
  (R-h_0)\int_{B_R\setminus B_{h_0}}\frac{|f(x)|}{|x|}\d\L^2(x)\\
  \leq & (R-h_0)\left(\int_{B_R\setminus
      B_{h_0}}|f|^2\d\L^2\right)^{1/2}\left(\int_{B_R\setminus
      B_{h_0}}|x|^{-2}\d\L^2(x)\right)^{1/2}\\
  \leq & (R-h_0)\left(\int_{B_R\setminus B_{h_0}}|f|^2\d\L^2\right)^{1/2}|\log
  h_0|^{1/2}
\end{split}
\]

Hence, writing $|u|:=|Dy^TDy-(D\ydel)^TD\ydel|$, we get
\[
\begin{split}
  \int_{h_0}^R|F(\rho)|\d\rho\leq & (R-h_0)\Bigg[\left(\int_{\partial
        B_{h_0}}|u|^2\d\H^1\right)^{1/2}\left(2\pi
      h_0\right)^{1/2}\\
      &+\left(\int_{B_R\setminus B_{h_0}}|u|^2\d
        \L^2\right)^{1/2}\left(\int_{B_R\setminus B_{h_0}}|x|^{-2}\d\L^2\right)^{1/2}\Bigg]\\
    &+\left(\int_{B_R\setminus B_{h_0}}|u|^2\d\L^2\right)^{1/2}(\pi
    (R^2-h_0^2))^{1/2}\\
    \leq & C h R|\log h|^{1/2}\,.
  \end{split}
\]
Furthermore, we  have the estimate
\[
\begin{split}
  \int_{h_0}^R|F''(\rho)|\d \rho=&\int_{B_R\setminus B_{h_0}}\left|\sum_{i=1}^3\det D^2 y_i\right|\d x  \\
  \leq &
  \int_{B_{1-5h}\setminus B_{5h}}|D^2 y|^2\d x\\
  \leq & C|\log h|\,.
\end{split}
\]

Now the claim of the proposition follows from the standard interpolation
inequality
\[
\|F'\|_{L^1(h_0,R)}\leq C \|F\|_{L^1(h_0,R)}^{1/2}\|F''\|_{L^1(h_0,R)}^{1/2}\,,
\]
see e.g.~\cite{gilbarg2001elliptic}, Theorem 7.28. This completes the proof.
\end{proof}

\begin{proposition}
\label{prop:diadicannuli}
Let $y_i\in C^2(B_1)$ for $i=1,2,3$. Assume that $\alpha>0$ and

\begin{equation}
\|\kappa_y-\pi\Delta^2\|_{L^1(6h,R)}\leq C h^{1/2}R^{1/2}|\log h|^\alpha\quad\text{for all }R\in[6h,1-5h]\,.\label{eq:28}
\end{equation}
Then 
\[
\sum_{i=1}^3 \int_{B_{1-5h}\setminus B_{5h}}|D^2y_i|^2\d x\geq 2\pi\Delta^2 \left(|\log h|-2\alpha \log|\log
  h|-C\right)\,.
\]
\end{proposition}
\begin{proof}
 We set
\[
h_0=5h|\log h|^{2\alpha}\,,
\]
and choose $J\in\N$ such that 
\[
2^Jh_0\leq 1-5h\leq 2^{J+1}h_0\,.
\]
Note that this choice implies $J\log 2-C\leq |\log h_0|\leq (J+1)\log 2+C$.
For $j=0,\dots,J+1$, we will write $R_j:=2^j h_0$.\\
We use Jensen's inequality to get for $r\in [5h,1-5h]$,
\begin{equation}
\int_{\partial B_r}|D^2y_i|^2\d \H^1(x)\geq  2\pi r\left(\frac{\int_{\partial
      B_r}|D^2 y_i|\d\H^1}{2\pi r}\right)^2\quad \text{for $i\in\{1,2,3\}$.}\label{eq:21}
\end{equation}
Recall the definition of $\kappa_y$ in \eqref{eq:14}.
We use Lemma \ref{lem:isoper} on the right hand side in \eqref{eq:21}, sum
over $i$, integrate over $r$, and use the triangle inequality 
to obtain
\[
  \sum_{i=1}^3\int_{B_{1-5h}\setminus B_{5h}}|D^2y_i|^2\d x
\geq \frac{1}{2\pi}
  \int_{5h}^{1-5h}\frac{\d r}{r}4\pi
    |\kappa_y(r)|\,.
\]
Reducing the domain of integration and using again the triangle inequality, we have

\begin{equation}
\begin{split}
  \sum_{i=1}^3\int_{B_{1-5h}\setminus B_{5h}}|D^2y_i|^2\d x
  \geq& 2\pi\Delta^2 \int_{R_0}^{R_J}\frac{\d r}{r}-2\int_{R_0}^{R_J}\frac{\d
    r}{r}|\kappa_y(r)-\pi\Delta^2|\\
  = &2\pi\Delta^2 \left(|\log h|-2\alpha\log|\log
    h|-C\right)\\
&-2\int_{R_0}^{R_J}\frac{\d r}{r}|\kappa_y(r)-\pi\Delta^2|\,.
\end{split}\label{eq:20}
\end{equation}
Using the assumption \eqref{eq:28}, we may estimate the last term on the right hand side  as an error term,
\[
\begin{split}
  \int_{R_0}^{R_J}\frac{\d r}{r}|\kappa_y(r)-\pi\Delta^2|= & \sum_{j=1}^{J}
  \int_{R_{j-1}}^{R_{j}}\frac{\d r}{r}|\kappa_y(r)-\pi\Delta^2|\\
  \leq & \sum_{j=1}^{J} C R_{j}^{-1} \|\kappa_y-\pi\Delta^2\|_{L^1(h,R^{j})}\\
  \leq & \sum_{j=1}^{J} C R_j^{-1/2}h^{1/2}|\log h|^{\alpha} \\
  \leq & C\,,
\end{split}
\]
where in the second inequality, we have used \eqref{eq:7}, and the last inequality is just the summation of a geometric series. Inserting this in \eqref{eq:20}, the proof of the proposition is complete.
\end{proof}

\begin{proof}[Proof of Theorem \ref{thm:main1}]
First note that the minimum is actually attained by some $y \in
W^{2,2}(B_1; \R^3)$, since the functional \eqref{eq:26} is coercive and lower
semi-continuous. The upper bound holds by Lemma \ref{lem:upperbound1}.
Assume that $y\in W^{2,2}(B_1;\R^3)$ with $\Ihd(y)\leq
2\pi\Delta^2h^2\left(|\log h|+C\right)$. By density of $C^2$ in $W^{2,2}$, we
may assume $y\in C^2(B_1;\R^3)$. By Proposition \ref{prop:interpol1},
the assumption of Proposition \ref{prop:diadicannuli} holds with
$\alpha=\frac34$. The lower bound now follows from that of Proposition \ref{prop:diadicannuli}.
\end{proof}

\section{Three-dimensional elasticity -- Proof of Theorem \ref{thm:3d}}
\label{sec:three-dimens-elast}

This section is concerned with the transition from two dimensions to
three. 
The three-dimensional model is defined by the free elastic energy functional 
\eqref{eq:29}. We first consider a special class of deformations. Let $y\in
C^2(B_{1,\Delta};\R^3)$ be an immersion, and let the three-dimensional
deformation $Y$ be given by the following Kirchhoff-Love ansatz:
\[
Y(x',x_3)= y(x')+x_3\nu_y\quad \text{ for }(x',x_3)\in B_{1,\Delta}\times[-h/2,h/2]\,,
\]
where $\nu_y=y_{,1}\wedge y_{,2}/|y_{,1}\wedge y_{,2}|$ denotes the unit
normal. With the assumption \eqref{eq:31} on the stored energy density $W$, and
provided that $\|Dy\|_{L^\infty}\leq C$, it is not difficult to show (cf.~the
proof of Lemma \ref{lem:upperbd3d} below)
that
\[
\fint_{[-h/2,h/2]}W(DY(x',x_3))\d x_3\leq C\left(\dist^2 (Dy(x'),O(2,3))+h^2|D^2y(x')|^2\right)\,,
\]
and hence 

\begin{equation}
\Ehd(Y)\leq C \int_{B_{1,\Delta}}\dist^2 (Dy,O(2,3))+h^2|D^2y|^2\d\L^2\,.\label{eq:19}
\end{equation}
Whenever a three-dimensional model and a plate model are related to each other
through this estimate for the Kirchhoff-Love ansatz, we
say for short that they correspond to each other.\\
\\
As a matter of fact, it is also possible to obtain lower bounds for three-dimensional models from
lower bounds in the corresponding plate models.
This is based on the
Geometric
Rigidity Theorem by Friesecke, James, M\"uller, which we cite now: 
\begin{theorem}{\cite{MR1916989}}
\label{thm:fjm}
Let $n\geq 2$ , and let $U\subset\R^n$  be a bounded connected Lipschitz
domain. There exists a constant $C^*=C^*(U)$ with the following property: For every $u\in
W^{1,2}(U;\R^n)$, there exists $R\in SO(n)$ such that
\[
\|Du-R\|_{L^2(U)}\leq C^*\int_U\dist^2(Du,SO(n))\d x\,.
\]
The constant $C^*$ is invariant under rescalings of $U$.
 \end{theorem}
The deduction of lower bounds  in three-dimensional elasticity starting from
lower bounds for the corresponding plate models via  Theorem \ref{thm:fjm}
has been carried out in \cite{Simon}. The main strategy of our proof is going to
be  similar to the one given in that reference. Since the estimates needed here
differ in many details from those in \cite{Simon}, we
nevertheless include a full proof here for the convenience of the reader.\\
The idea is as follows: From a given deformation in the three-dimensional
model, one constructs a deformation in the two-dimensional one through averaging
and smoothing. On small boxes of sidelength comparable to $h$ in the 3-dimensional
domain, the deformation gradient is close to $SO(3)$ by Theorem \ref{thm:fjm}.
This implies that the gradient of the 2-dimensional deformation derived from it
is close to $O(2,3)$. In this way, the membrane term of the two-dimensional
model can be estimated from above by the elastic energy in the 3-dimensional
one. The bending term can be dealt with in a similar way.\\
\\
We now recall some notation, and introduce some more.\\
\\
Recall from Section \ref{sec:statement-results} that
$\iota_\Delta:B_{1,\Delta}\to B_1$ is defined by
\[
\iota_\Delta(x)= |x|\left(\cos \frac{\varphi(x)}{\sqrt{1-\Delta^2}},\cos
  \frac{\varphi(x)}{\sqrt{1-\Delta^2}}\right)\,,
\]
where $\varphi(x)$ is the angular coordinate of $x$.
To alleviate the notation,  we write
\[\iota(x)\equiv \iota_\Delta(x)\]
for $x\in B_{1,\Delta}$. When convenient, we will use the notation
$z:=\iota(x)$.\\
On $\iota(B_{1,\Delta})=B_1\setminus \{(x_1,0):x_1\leq
0\}$, $\iota$ has a well defined inverse, that we denote by
\[
j: B_1\setminus \{(x_1,0):x_1\leq
0\}\to B_{1,\Delta}\,.
\]
For $y\in \A^{2d}$, let $\tilde y:\lambda_h\to\R^3$ be defined by  $\tilde y=y\circ j$.\\
In some places, it will be convenient to work, instead of $B_{1,\Delta}$, with
the smaller domain
\[
\ohd:=\left\{x\in B_{1-5h}\setminus
B_{5h}:0\leq \arccos\frac{x_1}{|x|} <\sqrt{1-\Delta^2}\pi\right\}\,.
\]
By slight abuse of notation, we will not distinguish between the map
$\iota$, defined on $B_{1,\Delta}$, and its restriction to $\ohd$. In the same
way, we will not distinguish between $j:B_1\setminus \R_-\to B_{1,\Delta}$ and its restriction
to $\lambda_h:= \iota(\ohd)=\left(B_{1-5h}\setminus B_{5h}\right)\setminus\{(x_1,0):x_1\leq
0\}$. \\
If for two quantities $f,g$, there exists a constant $C=C(\Delta)$ such that
\[
\frac1C f\leq g\leq Cf\,,
\]
we will write $f\simeq g$.\\
Furthermore, for tensors  $a\in \R^{3\times 2\times 2}$ and $b\in\R^{2\times 2}$,
let $a:(b,b)\in \R^{3\times 2\times 2}$ be defined by
\[
\left(a:(b,b)\right)_{ijk}:=\sum_{l,m=1}^2a_{ilm}b_{lj}b_{mk}
\]
for $i\in \{1,2,3\}$, $j,k\in\{1,2\}$.

\subsection{The corresponding plate theory}

To prove Theorem \ref{thm:3d}, we will first have to prove an energy scaling law
for a two-dimensional plate model that slightly differs from \eqref{eq:26},
namely, the one on the right hand side in \eqref{eq:19}. The
reason why we do not choose to work with the three-dimensional
model corresponding to \eqref{eq:26}  is that we want to apply  Theorem \ref{thm:fjm}, and the latter requires a flat
reference metric -- not only in the three-dimensional model, but in the
corresponding two-dimensional one too.\\
\\
Let 
\[
\A^{2d}:=\left\{y\in W^{2,2}_{\mathrm{loc.}}(B_{1,\Delta};\R^3):y\circ
  \iota^{-1}\in W^{2,2}(B_1;\R^3)\right\}\,.
\]
Our two-dimensional plate model is given by
$\Jhd:\A^{2d}\to \R$,
\[
\begin{split}
  \Jhd(y)  =&\int_{B_{1,\Delta}}|g_y-\id_{2\times 2}|^2+h^2|D^2y|^2\d\L^2\,.
\end{split}
\]
It will be convenient to define a slight variation of $\Jhd$, that consists of a
reduction of the domain of integration for the energy density. Let
\[
  \Jhdt(y)
  =\int_{\ohd}|g_y-\id_{2\times 2}|^2+h^2|D^2y|^2\d\L^2\,.
\]
\begin{lemma}
  \label{lem:cov}
With $\tilde y:=y\circ j$, we have for every $y\in \A^{2d}$, 
\begin{equation}
\begin{split}
  \Jhdt(y)\simeq 
  &\int_{\lambda_h}|g_{\tilde y}(z)-g_{\Delta}(z)|^2\\
  &+h^2\left|D^2\tilde y(z)-\Delta^2|z|^{-1}\left(D\tilde y(z)\cdot\hat z\right)\otimes
    \hat z^\bot\otimes \hat z^\bot\right|^2\d\L^2(z)\,,\\
  \Jhd(y)\simeq 
  &\int_{B_1}|g_{\tilde y}(z)-g_{\Delta}(z)|^2\\
  &+h^2\left|D^2\tilde y(z)-\Delta^2|z|^{-1}\left(D\tilde y(z)\cdot\hat z\right)\otimes
    \hat z^\bot\otimes \hat z^\bot\right|^2\d\L^2(z)\,.
\end{split}\label{eq:18}
\end{equation}

\end{lemma}
\begin{proof}
For $x\in \ohd$, we compute
\begin{equation}
\begin{split}
  D\iota|_x= &\hat z\otimes\hat
  x+\frac{1}{\sqrt{1-\Delta^2}}\hat z^\bot\otimes\hat x^\bot\\
  D^2 \iota|_x=&-\frac{\Delta^2}{1-\Delta^2}|x|^{-1}\hat z\otimes \hat x^\bot\otimes \hat x^\bot\,.
\end{split}\label{eq:46}
\end{equation}
The first equation implies
\begin{equation}
  \label{eq:45}
  D\iota|_x\simeq \id_{2\times 2}
\end{equation}
in the sense of positive definite matrices.
By the chain rule, $D y|_x=D\tilde y|_{\iota(x)}\cdot D\iota|_x$ for $y\in\A^{2d}$ and $x\in\ohd$.
As a consequence of the second equation in \eqref{eq:46},  we have
\begin{equation*}
\begin{split}
 D^2 y|_x
=&D^2 \tilde y|_{\iota(x)}:(D\iota|_x,D\iota|_x)+D\tilde y|_{\iota(x)}\cdot D^2\iota|_x\\
=&D^2 \tilde y|_{\iota(x)}:(D\iota|_x,D\iota|_x)
-
\frac{\Delta^2}{1-\Delta^2}|x|^{-1} \left(D\tilde y|_{\iota(x)}\cdot \hat z\right)\otimes \hat x^\bot\otimes \hat x^\bot\,.
\end{split}
\end{equation*}
Hence,
\begin{equation}
\begin{split}
  \Jhdt(y)=&\int_{\ohd}\left|D\iota|_x^TD\tilde y|_{\iota(x)}^TD\tilde
    y|_{\iota(x)}D\iota|_x-\id_{2\times 2}\right|^2\\
  &+h^2\Bigg|D^2\tilde
    y|_{\iota(x)}:(D\iota|_x,D\iota|_x)\\
&-\frac{\Delta^2}{1-\Delta^2}|x|^{-1}\left(D\tilde
      y|_{\iota(x)}\cdot \hat z\right)\otimes \hat x^\bot\otimes \hat
      x^\bot\Bigg|^2\d\L^2(x)\,.
  \end{split}\label{eq:48}
  \end{equation}
By \eqref{eq:45}, we have for every $a\in \R^{3\times 2\times 2}$, and for every
$b\in \R^{2\times 2}$,
\[
\begin{split}
  |a:(D\iota|_x^{-1},D\iota|_x^{-1})|^2\simeq &|a|^2\\
  \left|\left(D\iota|_x^{-1}\right)^T b\,D\iota|_x^{-1}\right|^2\simeq &|b|^2\,.
\end{split}
\]
Using this in \eqref{eq:48}, we get
\begin{equation}
\begin{split}
  \Jhdt(y)\simeq &\int_{\ohd}\left|D\tilde y|_{\iota(x)}^TD\tilde
    y|_{\iota(x)}-D\iota|_x^{-1}\left(D\iota|_x^{-1}\right)^T\right|^2\\
 & + h^2 \left|D^2\tilde y|_{\iota(x)}-\Delta^2\left(D\tilde y|_{\iota(x)}\cdot\hat
      z\right)\otimes \hat z^\bot\otimes \hat z^\bot\right|^2\d\L^2(x)\,.
\end{split}\label{eq:49}
\end{equation}
Multiplying this relation with a factor  $\left|\det
D\iota|_x\right|=(1-\Delta^2)^{-1/2}\simeq 1$, and using the coarea formula, we
obtain the first relation in \eqref{eq:18}; the second is obtained in exactly
the same way.
\end{proof}

\begin{lemma}
\label{lem:upperbd2}
  There exists $y^{h}\in \A^{2d}$ with 
\[
\begin{split}
  |Dy^h|\leq& C\\
  |D\nu_{y^h}|\leq &C |D^2y^h|\\
  \Jhd(y^h)\leq& C h^2|\log h|\,.
\end{split}
\]
\end{lemma}
\begin{proof}
   Let $\tilde y^h:B_1\to \R^3$ be the test function  that  we constructed in Lemma
\ref{lem:upperbound1}; i.e., for $x\in B_1$,
\[
\tilde y^h(x)=\eta(|x|/h) \ydel(x)\,,
\]
where
$\eta\in C^\infty([0,\infty))$ with $\eta(t)=0$ for $t\leq \frac12$, $\eta(t)=1$ for $t\geq
  1$,  $|\eta'|\leq 4$, $|\eta''|\leq 8$; and $\ydel$ is given by
\[
\ydel(x)= \sqrt{1-\Delta^2}x+\Delta e_3|x|\,.
\]
We set $y^h:=\tilde y^h\circ \iota$.
From the explicit formulas for $D\tilde y^h,D^2\tilde y^h$  in the proof of
Lemma \ref{lem:upperbound1} and \eqref{eq:46}, we see that $|D y^h|\leq C$, $ D 
y^h=0$ on $B_{h/2}$, and $|D\nu_{y^h}|\leq C|D^2 y^h(z)|\leq C |z|^{-1}$ for $z\in B_1\setminus
B_{h/2}$. Furthermore, we have
$g_{\tilde y^h}=\gd$ on
$B_{1}\setminus B_h$ and $|g_{\tilde y^h}|\leq C$ on $B_h$.
Hence, using \eqref{eq:18}, we may estimate 
\[
\begin{split}
  \Jhd(y^h)\simeq &\int_{B_{1}}|g_{\tilde y^h}(z)-\gd(z)|^2
+h^2\left|D^2 \tilde y^h(z)-\Delta^2|z|^{-1}(D
    \tilde y^h(z)\cdot \hat z)\otimes \hat z^\bot\otimes\hat z^\bot\right|^2\d\L^2(z)\\
\leq& \int_{B_h}C\d\L^2
+h^2 \int_{B_{1}\setminus B_{h/2}}
C|z|^{-2}\d\L^2(z)\\
    \leq & C h^2\left(|\log h|+1\right)\,.
\end{split}
\]
This proves the lemma.
\end{proof}

\begin{proposition}
\label{prop:Jhd}
For all $h$ small enough, we have
\[
\begin{split}
  \frac1C h^2|\log h|\leq &\inf_{y\in\A^{2d}} \Jhdt(y)\,.\\
\end{split}
\]
\end{proposition}
\begin{proof}
Assume that $\Jhdt(y)\leq Ch^2|\log h|$ (otherwise there is nothing to show). Setting $\tilde y:=y\circ j$, 
we may estimate
\begin{equation}
\begin{split}
  \int_{B_{1-5h}\setminus B_{5h}}|z|^{-2}&|D \tilde y(z)\cdot \hat
  z|^2\d\L^2(z)\\
\leq &
  \int_{B_{1-5h}\setminus B_{5h}}|z|^{-2}\d\L^2(z)+ \int_{B_{1-5h}\setminus
    B_{5h}}|z|^{-2}\left(|D \tilde y(z)\cdot \hat z|^2-1\right)\d\L^2(z)\\
\leq & 2\pi (|\log h|-C)\\
&+\left(\int_{B_{1-5h}\setminus
    B_{5h}}|z|^{-4}\d\L^2(z)\right)^{1/2}\left(\int_{B_{1-5h}\setminus
    B_{5h}}|g_{\tilde y}-\gd|^2\d\L^2(z)\right)^{1/2}\\
\leq & 2\pi |\log h|+C \left(h^{-2}\right)^{1/2} \left(h^2|\log
  h|\right)^{1/2}\\
\leq & 2\pi |\log h|+o(|\log h|),
\end{split}\label{eq:52}
\end{equation}
where we have used the Cauchy-Schwarz inequality and
Lemma \ref{lem:cov} in the second inequality.
Again by
Lemma \ref{lem:cov} and Young's inequality, we have
\[
\begin{split}
  h^2\int_{B_{1-5h}\setminus B_{5h}}|D^2\tilde y|\d\L^2\leq &C h^2\left(|\log
    h|+\int_{B_{1-5h}\setminus B_{5h}}|z|^{-2}|D\tilde y\cdot\hat z|^2\d\L^2(z)\right)\\
  \leq &Ch^2|\log h|\,.
\end{split}
\]
Hence, by Proposition \ref{prop:interpol1}, the assumption of Proposition
\ref{prop:diadicannuli} is fulfilled for $\tilde y$ with $\alpha=\frac32$. By
Proposition \ref{prop:diadicannuli}, 

\begin{equation}
\int_{B_{1-5h}\setminus B_{5h}}|D\tilde y|^2\geq 2\pi\Delta^2\left(|\log h|
  -\frac32\log|\log h|-C\right)\,.\label{eq:51}
\end{equation}
By \eqref{eq:52}, \eqref{eq:51} and the triangle inequality, we have
\[
\begin{split}
  \Big\|D^2\tilde y-&\Delta^2|z|^{-1}\left(D\tilde y\cdot \hat z\right)\otimes\hat z^\bot\otimes\hat
    z^\bot\Big\|_{L^2(B_{1-5h}\setminus B_{5h})}^2\\
\geq &\left(\sqrt {2\pi
      \Delta^2 |\log h|}-\sqrt {2\pi \Delta^4 |\log h|}\right)^2-o(|\log h|)\\
  =&2\pi\Delta(1-\Delta)|\log h|-o(|\log h|)\,.
\end{split}
\]
Combining this last estimate with Lemma \ref{lem:cov} proves the  proposition.
\end{proof}
\subsection{Reduction of domain and covers by boxes}
As we explained at the beginning of the present section, the core step in the passage
from two to three dimensions is to cover the three-dimensional domain by small boxes, and to
approximate the deformation by a Euclidean motion on each of these boxes. 
In this approach, one needs to take special care of the boundary of the
two-dimensional domain. Part of that boundary (namely, $\partial B_{1,\Delta}\cap
\partial B_1$) has already been taken care of by ``shrinking'' $B_{1,\Delta}$ to
$\ohd$ in the two-dimensional plate theory, as we did in the previous subsection.
The present subsection's aim is to control the boxes that are close to the
remaining part of the boundary. 
The proof  of this subsection's main statement  (Lemma \ref{lem:Qstar} below) is straightforward; however, we will need to introduce a certain amount of auxiliary
notation. \\
\\
We will be interested in coverings of $\ohd$ by small squares of sidelength
$h$ or $3h$. For $a\in \R^2$ and $r>0$, let $Q_{a,r}:=a+[-r/2,r/2]^2$. We define

\begin{equation}
\begin{split}
  \Q_h:=&\{a\in h\Z^2:Q_{a,h}\cap \ohd\neq \emptyset\}\\
 \Q_h^*:=&\{a\in\Q_h:Q_{a,3h}\not\subset B_{1,\Delta}\}\,.
\end{split}\label{eq:41}
\end{equation}
Furthermore, let
\[
\tilde
B_{1,\Delta}:=\left\{x\in B_1\setminus\{0\}:\left(1-\sqrt{1-\Delta^2}\right)\pi<\arccos
\frac{x_1}{|x|}\leq \pi\right\}\,.
\]
Now we will define maps $f_\Delta^{(\pm)}:\tilde B_{1,\Delta}\to B_{1,\Delta}$ that
translate the cut in the domain in angular direction. 
Let
$\phi_\Delta:=(1-\sqrt{1-\Delta^2})2\pi$ and let the rotation $S_\Delta\in
SO(2)$ be defined by
\[
S_{\Delta}=\left(\begin{array}{cc}\cos\phi_\Delta & -\sin\phi_\Delta\\
\sin\phi_\Delta&\cos\phi_\Delta \end{array}\right)\,.
\]
Let 
\[
v^{\pm}:=\left(\cos \sqrt{1-\Delta^2} \pi,\pm\sin
\sqrt{1-\Delta^2}\pi\right)\,.
\]
Note that $v^\pm$ are chosen such that $[0,v^+)\cup[0,v^-)=\partial
\tilde B_{1,\Delta}\setminus\partial B_1$, where by $[a,b)$, we denote the line segment
connecting $a$ with  $b$ in $\R^2$, including $a$, but excluding
$b$. Furthermore, note that $S_{\Delta}v^-=v^+$.\\
Now let
\[
\begin{split}
  \tilde B_{1,\Delta,+}:=&\{x\in \tilde B_{1,\Delta}:x\cdot (v^+)^\bot\geq 0\}\\
  \tilde B_{1,\Delta,-}:=&\{x\in \tilde B_{1,\Delta}:x\cdot (v^-)^\bot\leq 0\}
\end{split}
\]
and define $f_{\Delta,\pm}$ by
\[
\begin{split}
  f_{\Delta,+}(x)=&\begin{cases} S_\Delta^{-1} x &\text{ if }x\in \tilde B_{1,\Delta,+}\\
    x &\text{ else}\end{cases}\\
   f_{\Delta,-}(x)=&\begin{cases} S_\Delta x &\text{ if }x\in \tilde B_{1,\Delta,-}\\
     x &\text{ else}\end{cases}
\end{split}
\]
For a sketch of $\tilde B_{1,\Delta,\pm}$ and $f_{\Delta,\pm}$, see Figure \ref{fig:b1d}.

\begin{figure}[h]
\begin{center}
\includegraphics[width=.7\textwidth]{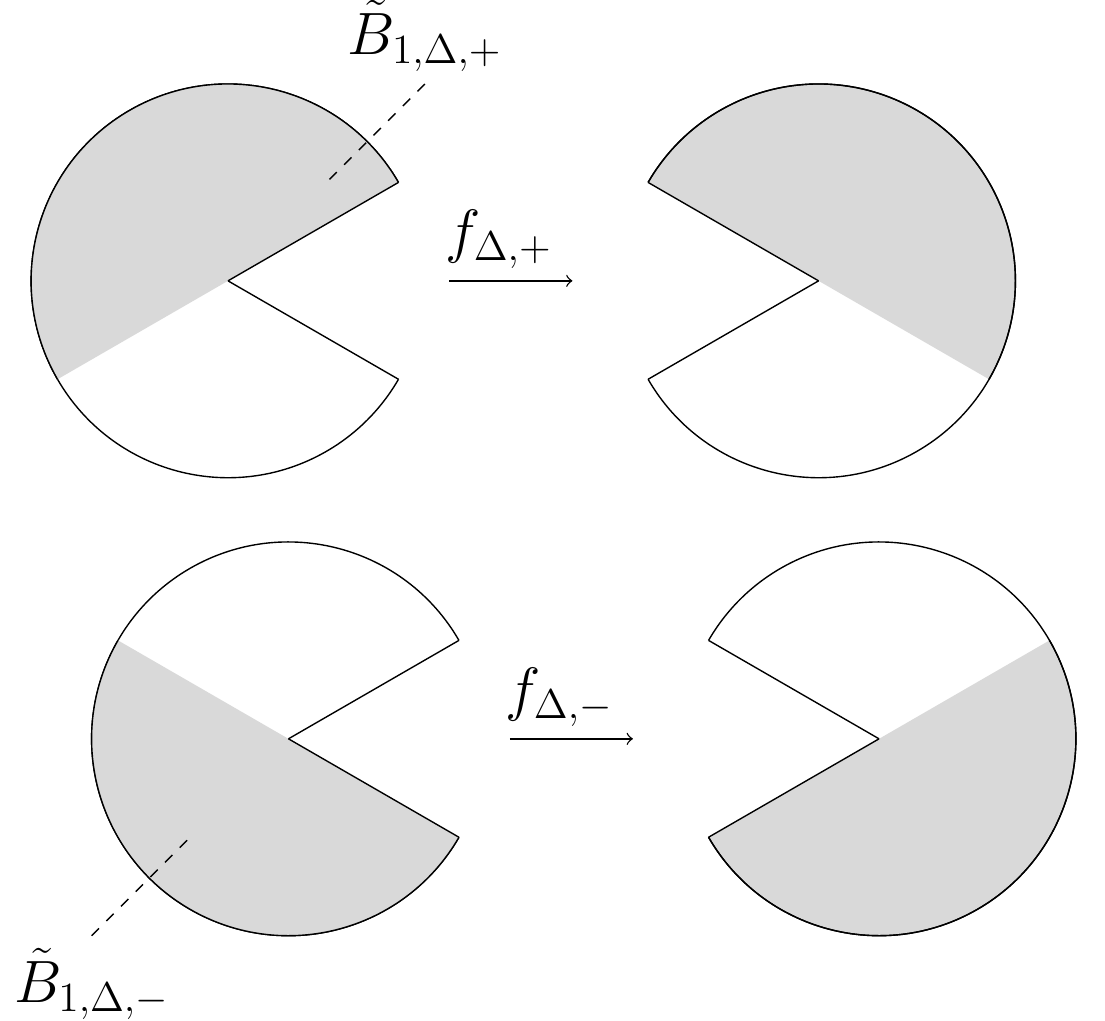}
\caption{The maps $f_{\Delta,\pm}$. Colored in gray: in the top right, the set
  $\tilde B_{1,\Delta,+}$;
in the top right,  $f_{\Delta,+}(\tilde B_{1,\Delta,+})$; in the bottom left, $B_{1,\Delta,-}$;
in the bottom right,  $f_{\Delta,-}(\tilde B_{1,\Delta,-})$.\label{fig:b1d}}
\end{center}
\end{figure}

Note that for $x\in \tilde B_{1,\Delta}$,

\begin{equation}
f^{-1}_{\Delta,-}\circ f_{\Delta,+}(x)=\begin{cases}S_{\Delta}^{-1}x &\text{ if }
  S_{\Delta}^{-1}x \in \tilde B_{1,\Delta}\\S_{\Delta}^{-2}x&\text{
    else.}\end{cases}\label{eq:43}
\end{equation}
For $Y\in \Ahd$, we define $Y^{(\pm)}:\tilde B_{1,\Delta}\times[-h/2,h/2]\to\R^3$ by

\begin{equation}
Y^{(\pm)}(x',x_3)=Y(f_{\Delta,\pm}(x'),x_3)\,.\label{eq:40}
\end{equation}

The main statement of this subsection is
\begin{lemma}
\label{lem:Qstar}
Let $Y\in \Ahd$. Then
  \begin{equation}
    \begin{split}
      \sum_{a\in
        \Q_h^*}\int_{Q_{a,3h}\times[-h/2,h/2]}&\dist^2(DY^{(\pm)},SO(3))\d\L^3\\
&\leq
      3\int_{
        B_{1,\Delta}\times[-h/2,h/2]}\dist^2(DY,SO(3))\d\L^3\,.\label{eq:39}
    \end{split}
\end{equation}

\end{lemma}
\begin{proof}
First, we note that as a direct consequence of the definition of $\Q_h^*$, we
have that
  for every $a\in\Q_h^*$, $Q_{a,3h}\subset \tilde B_{1,\Delta}$.
Next,  
for $g\in L^1(B_{1,\Delta})$, let
\[
g_{\Delta,\pm}=g\circ f_{\Delta,\pm}\,.
\]
Note that away from their discontinuity sets (which are $\L^2$ null sets), $f_{\Delta,\pm}$ are isometries. Hence, we have
\begin{equation}
\begin{split}
  \sum_{a\in \Q_h^*}\int_{Q_{a,3h}}|g_{\Delta,\pm}|\d\L^2\leq&
  3\int_{\tilde B_{1,\Delta}}|g_{\Delta,\pm}|\d\L^2\,\\
  = & 3\int_{ B_{1,\Delta}}|g|\d\L^2\,.
\end{split}\label{eq:37}
\end{equation}
For every $x'\in \tilde B_{1,\Delta}$ that is not on the discontinuity
set of $f_{\Delta,\pm}$, there exists $R\in SO(3)$ such that we have
\[
DY^{(\pm)}(x',x_3)=DY(f_{\Delta,\pm}(x'),x_3) R\,, 
\]
and hence for every $(x',x_3)\in \tilde B_{1,\Delta}\times [-h/2,h/2]$,
\begin{equation}
\dist^2(DY^{(\pm)}(x',x_3),SO(3))=\dist^2(DY(f_{\Delta,\pm}(x'),x_3),SO(3))\,.\label{eq:38}
\end{equation}
For fixed $x_3\in [-h/2,h/2]$, combining \eqref{eq:37}, \eqref{eq:38} and the fact that the discontinuity set
of  $f_{\Delta,\pm}$ is an $\L^2$ null set, yields 
\[
\sum_{a\in\Q_h^*}\int_{Q_{a,3h}}\dist^2(DY^{(\pm)}(x',x_3),SO(3))\d\L^2( x')\leq
3\int_{B_{1,\Delta}}\dist^2(DY(\tilde x',x_3),SO(3))\d\L^2( \tilde x')\,,
\]
where we have used the coarea formula to make the change of variables
$\tilde x'=f_{\Delta\pm}(x')$ on the right hand side.
Integrating in $x_3$ yields
the claim of the lemma.
\end{proof}


\subsection{Proof of Theorem \ref{thm:3d}}
\begin{lemma}
\label{lem:upperbd3d}
 We have
\[
 \inf_{Y\in\Ahd}\Ehd(Y)\leq C h^2|\log h|\,.
\]
\end{lemma}
\begin{proof}
  Let $y^h:\ohd\to\R^3$ be as in the statement of Lemma \ref{lem:upperbd2}.
Now, for $(x',x_3)\in\Ohd$,  we set
\[
Y^h(x',x_3)= y^h(x')+x_3\nu_{y^h}(x')\,.
\]
We compute
\[
DY^h(x',x_3)=Dy^h(x')+\nu_{y^h}(x')\otimes e_3+ x_3 D\nu_{y^h}(x')\,.
\]
By definition of the unit normal $\nu_{y^h}$, we have
\[
\dist^2(Dy^h(x')+\nu_{y^h}(x')\otimes e_3,SO(3))\leq
\dist^2(Dy^h(x'),O(2,3))\,.
\]
By Lemma \ref{lem:upperbd2}, we have $|D\nu_{y^h}|^2\leq C |D^2y^h|^2$,
and hence, we get by integrating in $x_3$,
\begin{equation}
  \begin{split}
    \fint_{[-h/2,h/2]}\dist^2&(DY^h(x',x_3),SO(3))\d x_3\\
&\leq
    C\left(\dist^2(Dy^h(x'),O(2,3))+h^2|D^2y^h(x')|^2\right)\,\label{eq:53}
  \end{split}
\end{equation}
for all $x'\in \ohd$.\\
Since $W(F)<C\dist^2(F,SO(3))$ in a neighborhood of $SO(3)$ by \eqref{eq:31}, we have that for
any compact set $\mathcal F\subset \R^{3\times 3}$, there exists a constant $C=C(W,\mathcal
F)$ such that
\[
W(F)<C\dist^2(F,SO(3)) \quad\text{ for }F\in\mathcal F\,.
\]
Since $|Dy^h|\leq C$, we get 
\[
W(DY^h)<C(W)\dist^2(Dy^h,SO(3))\,.
\]
Hence, using \eqref{eq:53},
\[
\begin{split}
  \Ehd(Y^h)=&\int_{B_{1,\Delta}}\d x'\fint_{[-h/2,h/2]}\d x_3 W(DY^h)\\
  \leq &C(W)\int_{B_{1,\Delta}}\d x'\left(\dist^2(Dy^h,O(2,3))+h^2|D^2y^h|^2\right)\\
\leq & C(W) h^2|\log h|\,,
\end{split}
\]
where the last estimate holds by Lemma \ref{lem:upperbd2}.
This proves the  lemma.
\end{proof}

\begin{proof}[Proof of Theorem  \ref{thm:3d}]
The upper bound has been proved in Lemma
\ref{lem:upperbd3d}, while the existence of a minimizer follows from coercivity
and lower semicontinuity of $\Ehd$ in $W^{1,2}(B_{1,\Delta}\times [-h/2,h/2];\R^3)$. It remains to show the lower bound.\\
{\bf Step 1.} Extension of the domain and mollification.
Assume  that $Y\in W^{1,2}(B_{1,\Delta}\times[-h/2,h/2];\R^3)$. We define $y^{(1)}:B_{1,\Delta}\to\R^3$ by setting
\[
y^{(1)}(x'):=\fint_{[-h/2,h/2]} Y(x',x_3)\d x\,.
\]
For each $x'\in\tilde B_{1,\Delta}$, we set
\[
y^{(1,\pm)}(x'):= y^{(1)}(f_{\Delta,\pm}(x'))\,.
\]
Next, we mollify; let $\eta\in C^\infty_0(\R^2)$ with
$\int_{\R^2}\eta=1$ and $\supp\eta\subset B_1$. We set
$\eta_h(\cdot):=h^{-2}\eta(\cdot/h)$. Now, for  $x'\in \bigcup_{a\in\Q_h\setminus\Q_h^*}Q_{a,h}$, we set
\[
y^{(2)}(x')=(\eta_h*y^{(1)})(x')\,.
\]
For $x'\in \tilde B_{1,\Delta}$, we define
\[
y^{(2,\pm )}(x')=(\eta_h*y^{(1,\pm )})(x')\,.
\]
Since linear transformations and convolution commute, we have by \eqref{eq:43},
\begin{equation}
y^{(2,+)}(x')=y^{(2,-)}(S_\Delta^{-1} x')\quad \text{ for all } x'\in \tilde
B_{1,\Delta}\text{ with } S_\Delta^{-1}x'\in \tilde B_{1,\Delta}\,.\label{eq:42}
\end{equation}
\newcommand{\ohdp}{\omega_{h,\Delta,+}}
\newcommand{\ohdm}{\omega_{h,\Delta,-}}
\newcommand{\ohdpm}{\omega_{h,\Delta,\pm}}
Now let
\[
\ohdpm=\left\{x\in \bigcup_{a\in\Q^*_h}Q_{a,h}\cap\ohd:f_{\Delta,\pm}(x')=x' \right\}\,.
\]
Note that we have 
\[
\ohd= \ohdp\cup\ohdm\cup\bigcup_{a\in\Q_h\setminus\Q_h^*}Q_{a,h}\,,
\]
where the sets on the right hand side are mutually disjoint.
On $\ohdp\cup\ohdm$, we define $y^{(2)}$ by
\[
y^{(2)}(x')=y^{(2,\pm)}(x') \quad\text{ if } x'\in \ohdpm\,.
\]
By \eqref{eq:42}, we have
\[
y^{(2)}\circ j\in W^{2,2}(B_{1-5h}\setminus B_{5h};\R^3)\,.
\]
{\bf Step 2.} Application of Theorem \ref{thm:fjm} on small boxes.
Let $a\in \Q_h\setminus \Q_h^*$. 
By Theorem \ref{thm:fjm},  there exists $R_{a,h}\in SO(3)$ such that
\begin{equation}
\int_{Q_{a,3h}\times[-h/2,h/2]}\dist^2 (DY,SO(3))\d x\leq C
\int_{Q_{a,3h}\times[-h/2,h/2]}|DY-R_{a,h}|^2\d x\,.\label{eq:8}
\end{equation}
For  $a\in  \Q_h^*$, there exist $R_{a,h}^{(\pm)}\in SO(3)$ such that
\begin{equation}
\int_{Q_{a,3h}\times[-h/2,h/2]}\dist^2 (DY^{(\pm)},SO(3))\d x\leq C
\int_{Q_{a,3h}\times[-h/2,h/2]}|DY^{(\pm)}-R_{a,h}^{(\pm)}|^2\d x\,,\label{eq:8}
\end{equation}
where $Y^{(\pm)}$ has been defined in \eqref{eq:40}.
For $a\in\Q_h\setminus \Q_h^*$,  let the 3 by 2 matrix consisting of the first two columns of
$R_{a,h}$ be
denoted by $\hat R_{a,h}$. For $a\in \Q_h$, let $\hat R_{a,h}^{(\pm)}$ denote the
first two columns of $R_{a,h}^{(\pm)}$ respectively.\\
{\bf Step 3.} Conclusion.
Note that for every $a\in \Q_h\setminus \Q_h^*$, and every $x'\in Q_{a,3h}$,
\[
\left|Dy^{(1)}(x')-\hat R_{a,h}\right|\leq \fint_{[-h/2,h/2]} \left|DY(x',x_3)-R_{a,h}\right|\d x_3 \,.
\]
An analogous formula holds for $a\in \Q_h^*$.
Now we have

\begin{equation}
\begin{split}
  \int_{\ohd}\dist^2(Dy^{(2)},O(2,3))\d\L^2
  \leq & \sum_{a\in\Q_h\setminus \Q_h^*} \int_{Q_{a,h}}\dist^2(Dy^{(2)},O(2,3))\d \L^2\\
&+ \sum_{a\in\Q_h^*,i\in\{\pm \}} \int_{Q_{a,h}}\dist^2(Dy^{(2,i)},O(2,3))\d \L^2\\
\leq & \sum_{a\in\Q_h\setminus \Q_h^*} \int_{Q_{a,h}}|Dy^{(2)}-\hat R_{a,h}|^2\d \L^2\\
&+ \sum_{a\in \Q_h^*,i\in\{\pm \}} \int_{Q_{a,h}}|Dy^{(2,i)}-\hat R_{a,h}^{(i)}|^2\d \L^2
\end{split}\label{eq:30}
\end{equation}
We use the definition of $y^{(2)}$ and Young's inequality for convolutions to get (for $a\in\Q_h\setminus\Q_h^*$)
\[
\begin{split}
\int_{Q_{a,h}}|Dy^{(2)}-\hat R_{a,h}|^2\d \L^2\leq& \int_{Q_{a,h}}|\eta_h*(Dy^{(1)}-\hat R_{a,h})
|^2\d \L^2\\
\leq & C\int_{Q_{a,3h}}|Dy^{(1)}-\hat R_{a,h}|^2\d \L^2\,,
\end{split}
\]
with   analogous estimates for $a\in \Q_h^*$. Inserting in \eqref{eq:30}, we get

\begin{equation}
\begin{split}
  \int_{\ohd}&\dist^2(Dy^{(2)},O(2,3))\d \L^2 \\
\leq &C\Bigg(\sum_{a\in\Q_h\setminus \Q_h^*}\int_{Q_{a,3h}}|Dy^{(1)}-\hat R_{a,h}|^2\d \L^2\\
&+ \sum_{a\in\Q_h^*,i\in\{\pm \}}\int_{Q_{a,3h}}|Dy^{(1,i)}-\hat
R_{a,h}^{(i)}|^2\d \L^2\Bigg)\\
\leq & Ch^{-1}\Bigg(\sum_{a\in\Q_h\setminus \Q_h^*}\int_{Q_{a,3h}\times[-h/2,h/2]}|DY- R_{a,h}|^2\d\L^3\\
&+ \sum_{a\in\Q_h^*,i\in\{\pm \}}\int_{Q_{a,3h}\times[-h/2,h/2]}|DY^{(i)}-
R_{a,h}^{(i)}|^2\d \L^3\Bigg)\\
\leq & Ch^{-1}\Bigg(\sum_{a\in\Q_h\setminus \Q_h^*}\int_{Q_{a,3h}\times[-h/2,h/2]}\dist^2(DY,SO(3))\d \L^3\\
&+ \sum_{a\in\Q_h^*,i\in\{\pm \}}\int_{Q_{a,3h}\times[-h/2,h/2]}\dist^2(DY^{(i)},SO(3))\d \L^3\Bigg)\,.
\end{split}\label{eq:33}
\end{equation}
Using  Lemma \ref{lem:Qstar}  in \eqref{eq:33}, we get

\begin{equation}
\int_{\ohd}\dist^2(Dy^{(2)},O(2,3))\d \L^2\leq Ch^{-1}\int_{B_{1,\Delta}\times
  [-h/2,h/2]}\dist^2 (DY,SO(3))\d \L^3\,.\label{eq:34}
\end{equation}

We come to the estimate of the bending term, using again the definition of $y^{(2)}$
and Young's inequality for convolutions:

\begin{equation}
\begin{split}
  \int_{\ohd}|D^2y^{(2)}|^2\d\L^2 \leq &
  \sum_{a\in\Q_h\setminus\Q_h^*}\int_{Q_{a,h}}\left|D\eta_h*(Dy^{(1)}-\hat
    R_{a,h})\right|^2\d\L^2\\
&+\sum_{a\in\Q_h^*,i\in\{\pm \}}\int_{Q_{a,h}}\left|D\eta_h*(Dy^{(1,i)}-\hat
    R_{a,h}^{(i)})\right|^2\d\L^2\\
  \leq & Ch^{-2}\Bigg(\sum_{a\in\Q_h\setminus\Q_h^*}\int_{Q_{a,3h}}\left|Dy^{(1)}-\hat
    R_{a,h}\right|^2\d\L^2\\
&+\sum_{a\in\Q_h^*,i\in\{\pm \}}\int_{Q_{a,3h}}\left|Dy^{(1,i)}-\hat
    R_{a,h}^{(i)}\right|^2\d\L^2\Bigg)\,.
\end{split}\label{eq:35}
\end{equation}
Using the  chain of inequalities from \eqref{eq:33} starting from the second
one, and again Lemma \ref{lem:Qstar}, we get from \eqref{eq:35},
\begin{equation}
\int_{\ohd}|D^2y^{(2)}|^2\d x \leq 
Ch^{-3}\int_{B_{1,\Delta}\times[-h/2,h/2]}\dist^2(DY,SO(3))\d\L^3\,.\label{eq:36}
\end{equation}
Combining \eqref{eq:34} and \eqref{eq:36}, we get
\[
 \Jhd(y^{(2)})\leq Ch^{-1} \int_{B_{1,\Delta}\times[-h/2,h/2]}\dist^2(DY,SO(3))\d\L^3\,.
\]
By \eqref{eq:32}, this implies
\[
 \Jhd(y^{(2)})\leq Ch^{-1} \int_{B_{1,\Delta}\times[-h/2,h/2]}W(DY)\d\L^3=C\Ehd(Y)\,.
\]
By Proposition \ref{prop:Jhd},
this proves the lower bound, and completes
the proof of the theorem.
\end{proof}

\section{The \FvK setting -- Proof of Theorem \ref{thm:main}}
\label{sec:fvk-setting-proof}
The \FvK (FvK) plate model is a very popular model for phenomena
that include large deflections. Strictly speaking, there exists  no rigorous justification
for the validity of the model when large deflections occur. The FvK model can be derived from
three-dimensional nonlinear elasticity as a low energy Gamma-limit, see \cite{MR1916989,MR2210909}.
This rigorously justifies the FvK model, but only as a model for
 small deflections.\\
Nevertheless, the FvK equations, first formulated more than hundred
years ago \cite{von1910festigkeitsprobleme}, have a long and quite successful
history of describing large deformation phenomena in engineering, from the
design of submarine hulls over paper crumpling
\cite{MR2023444,audoly2010elasticity,Lobkovsky01121995,1997PhRvE..55.1577L} and wrinkling to the mechanical properties of cell walls
\cite{PhysRevA.38.1005,lidmar2003virus}. 
In the context of stability and buckling,
several interesting points on why the FvK equations are relevant, even in
regimes that they seemingly do not describe, have been made in \cite{MR2525119}.\\
In the applications mentioned above, the
FvK model usually still contains a small parameter $h$ that is
interpreted as the (rescaled) thickness of the sheet, whereas in the Gamma-Limit
from \cite{MR2210909}, the thickness has disappeared from the model. Strictly
speaking, it is  unclear if  plate models that still contain the small
parameter $h$ have a rigorous meaning asymptotically.\\
\\
Theorem \ref{thm:main} identifies the energy scaling of  a single disclination
in the FvK model that includes the small parameter $h$. The size of the disclination is of the
same order as the ``natural'' deflection size in the FvK model, and is
determined by the parameter $\Delta$. \\
\\
To see the relation between $\Ihd$ and $\vk_{h,\Delta}$, let $\ve$ be a small parameter, and set $\tilde \Delta=\ve\Delta$, $\tilde h=\ve h$. Assume
further that $(u,v)\in C^\infty(B_1; \R^2)\times C^\infty(B_1)$, and that the  deformation $y:B_1\to\R^3$ is given by
\[
y(x)=x+\ve^2 u+\ve v e_3\,.
\]
Now we have 
\[
\begin{split}
  Dy=&e_1\otimes e_1+e_2\otimes e_2+\ve^2 Du+\ve e_3\otimes Dv\\
  Dy^TDy=& \id_{2\times 2}+\ve^2 (Du+Du^T+Dv\otimes Dv)+O(\ve^3)\\
Dy^TDy-g_{\tilde \Delta}=& \ve^2\left(Du+Du^T+Dv\otimes Dv+\Delta^2\hat
x^\bot\otimes\hat x^\bot\right)\,,
\end{split}
\]
and hence 
\[
\begin{split}
  I_{\tilde h,\tilde \Delta}(y)= &\int_{B_1}|g_y-g_{\tilde \Delta}|^2+\tilde
  h^2|D^2y|^2\d\L^2\\
  =& \ve^4 \int_{B_1} |2\,\sym Du+Dv\otimes Dv+\Delta^2\hat
x^\bot\otimes\hat x^\bot|^2+h^2|D^2v|^2\d\L^2+O(\ve^5)\,.
\end{split}
\]
Hence, we see that $\vk_{h,\Delta}(u,v)$ is the lowest order term in $\ve$ of $I_{\ve
  h,\ve\Delta}(x\mapsto x+\ve^2 u+\ve v e_3)$.
\\
We remark  that the above expansion in $\ve$ can not be used
to deduce Theorem \ref{thm:main} from Theorem \ref{thm:main1} in a trivial way.
Nevertheless, the proof of Theorem \ref{thm:main} is very similar to the one of Theorem
\ref{thm:main1}.  The role that was played by the control variable 
$\kappa_y(r)=\int_{B_r}\sum_i\det D^2 y_i\d x$ there will be played by 
\begin{equation}
\kvk_v(r):= \int_{B_r}\det D^2 v\d x\label{eq:22}
\end{equation}
here. We recall that $\det D^2 v$ is the \FvK
version of Gauss curvature -- so our control variable is a curvature integral,
and by Gauss' equation may be thought of as the (oriented) volume of balls $B_r$
under the push-forward of the Gauss map.\\
The only significant difference between the two cases is that we will have to
identify the right terms in the membrane energy to make estimates for $\kvk_v$
in the $W^{-1,1}$ norm.\\
\\
We start off with the upper bound:

\begin{lemma}
\label{lem:upperbound2}
We have
  \begin{equation*}
  \inf_{(u,v)\in\A} \vk_{h,\Delta}(u,v)\leq 2\pi \Delta^2\,h^2|\log h|+Ch^2\,.
\end{equation*}

\end{lemma}
\begin{proof}
  Let  $\eta\in C^\infty([0,\infty))$ with $\eta(x)=0$ for $x\leq \frac12$, $\eta(x)=1$ for $x\geq
  1$ and  $|\eta'|\leq 4$, $|\eta''|\leq 8$. 
  Now we define $u^h:B_1\to\R^2$, $v^h:B_1\to\R$
  by 
\[
\begin{split}
  u^h(x):=&-\frac{\Delta^2}{2} \,\eta(|x|/h) x\\
  v^h(r,\varphi):=&\Delta \,\eta(|x|/h) |x|\,.
\end{split}
\]
The membrane energy density $\left|2\,\sym Du^h+Dv^h\otimes Dv^h+\Delta^2\hat
x^\bot\otimes \hat x^\bot\right|^2$ vanishes on
$B_1\setminus B_h$, while it is bounded by a constant $C$ that is independent of
$h$ on $B_h$. Hence
\[
E_{\mathrm{membrane}}:=\int_{B_1}\left|2\,\sym Du^h+Dv^h\otimes Dv^h+\Delta^2\hat
x^\bot\otimes \hat x^\bot\right|^2\d x\leq
Ch^2\,.
\]
Furthermore,
\[
D^2
v^h=\Delta\left(\frac{|x|}{h^2}\eta''(|x|/h)+\frac{2}{h}\eta'(|x|/h)\right)\hat x\otimes
\hat x+
\frac{\Delta}{|x|}\left(\frac{|x|}{h}\eta'(|x|/h)+\eta(|x|/h)\right)\hat x^\bot\otimes
\hat x^\bot\,.
\]
Hence
\[
\begin{split}
  |D^2 v^h|^2\leq &Ch^{-2} \quad\text{ for }|x|\leq Ch^{-2}\,, \\
|D^2 v^h|^2=&\Delta^2 |x|^{-2}\quad
  \text{ for }|x|> h,
\end{split}
\]
and 
\[
\begin{split}
  h^{-2}E_{\mathrm{bending}}:=\int_{B_1}|D^2 v^h|^2\d x\leq &\int_{B_h}Ch^{-2}\d x +2\pi
  \int_h^1\Delta^2\frac{\d r}{r}\\
  \leq & 2\pi\Delta^2|\log h|+C\,.
\end{split}
\]
This proves the lemma. 
\end{proof}

We come to the interpolation between metric and curvature that will yield an
$L^1$ estimate for $\kvk_v$.

\begin{proposition}
\label{prop:interpolvk}
  Let $(u,v)\in C^2(B_1;\R^3)$  with 
$\vk_h(y)\leq 2\pi \Delta^2 h^2(|\log h|+C)$. Then
\begin{equation}
\|\kvk_v-\pi\Delta^2\|_{L^1(h,R)}\leq C h^{1/2}R^{1/2}|\log h|\,.\label{eq:7}
\end{equation}
\end{proposition}
\begin{proof}
In the present proof, we will use polar coordinates $r,\varphi$ on $B_1$. The
unit vectors in $r$ and $\varphi$ direction are denoted by $e_r=\hat x$ and
$e_\varphi=\hat x^\bot$ respectively. The
vector field $u$ will be written as $u=u_re_r+u_\varphi e_\varphi$.
The membrane energy in these coordinates is given by
\[
\begin{split}
  \int_0^1r\d r\int_0^{2\pi}\d\varphi & \Bigg(\left|2u_{r,r}+v_{,r}^2\right|^2
  +\left|2r^{-1}(u_{,\varphi}^\varphi+u_r)+ (r^{-1}v_{,\varphi})^2+\Delta^2\right|^2\\
  &+2\left|u_{\varphi,r}+r^{-1}\left(u_{r,\varphi}-u_\varphi+v_{,r}v_{,\varphi}\right)\right|^2\Bigg)\,.
\end{split}
\]
  For $r>0$, we have
  \begin{equation}
    \begin{split}
      \int_{B_r}\det D^2 v\d x
      =\frac12 \int_{\partial B_r} v_{,1}\d v_{,2}-v_{,2}\d v_{,1}\,.
    \end{split}\label{eq:4}
  \end{equation}

  With
  $v_{,1}=v_{,r}\cos\varphi-r^{-1}v_{,\varphi}\sin\varphi$ and
  $v_{,2}=v_{,r}\sin\varphi+r^{-1}v_{,\varphi}\cos\varphi$, we
  compute
  
  \begin{equation}
  \begin{split}
    v_{,1}\d v_{,2}-v_{,2}\d v_{,1}=&\alpha \d r+\left(v_{,r}^2+\left(\frac{v_{,\varphi}}{r}\right)^2 + \frac{v_{,r}v_{,\varphi\varphi}}{r}-\frac{v_{,\varphi}v_{,r\varphi}}{r}\right)\d\varphi\\
    =& \alpha    \d r+\left(\left(v_{,r}\right)^2-\left(\frac{v_{,\varphi}^2}{r}\right)_{,r}+\frac{\left(v_{,r}v_{,\varphi}\right)_{,\varphi}}{r}\right)\d\varphi\,,
  \end{split}\label{eq:9}
  \end{equation}
where $\alpha$ is some function of $r,\varphi$ that will be irrelevant for
our purpose.
  For $h\leq s,r\leq 1$, let
  \[
  \begin{split}
    a(s):=&\int_h^s \d r\int_0^{2\pi}\d\varphi \left(2u_{r,r}+ v_{,r}^2\right)\\
    b(r):=&\int_0^{2\pi}\d\varphi \left(2u_r+
      \frac{v_{,\varphi}^2}{r}+\Delta^2 r\right)
  \end{split}
  \]
  Now we set $F(s):=\frac12 \left(a(s)-b(s)\right)$, and compute
  \[
  F'(r)=\frac12\int\d \varphi\left( v_{,r}^2-
    \left(\frac{v_{,\varphi}^2}{r}\right)_{,r}-\Delta^2\right)\,.
  \]
Comparing this last expression with \eqref{eq:4} and \eqref{eq:9}, we have

\begin{equation}
  \begin{split}
    F'(r)=&- \pi\Delta^2+\int_{B_r}\det D^2 v\d x \\
  F''(r)=&\int r\d \varphi \det D^2v\,.
\end{split}\label{eq:17}
\end{equation}

  This gives us the needed bound on $\|F''\|_{L^1}$:
  \begin{equation}
    \label{eq:6}
    \begin{split}
      \|F''\|_{L^1(h,1)}\leq &\int_0^1r\d r\int_0^{2\pi}\d\varphi |\det
      D^2
      v|\\
      \leq & \int_{B_1}|D^2 v|\d x\\
      \leq & C|\log h|\,.
    \end{split}
  \end{equation}


 The $L^1$ bound on
  $F$ works as follows,
  \[
  \begin{split}
    \int_h^R |b(r)|\d r=&\int_h^R \d r\left|\int\d\varphi \left(2u_r+ \frac{v_{,\varphi}^2}{r}+\Delta^2 r\right)\right|\\
    \leq & \left(\int_h^1 r\d r\d\varphi
      \left|2r^{-1}(u_{\varphi,\varphi}+u_r)+ (r^{-1}v_{,\varphi})^2+\Delta^2\right|^2\right)^{1/2}\left(\int_h^R r\d r\right)^{1/2}\\
    \leq& C h|\log h|^{1/2}R\,,
  \end{split}
  \]
  and
  \[
  \begin{split}
    |a(s)|=\left|\int_h^s\d r \int\d\varphi \left(2u_{r,r}+
        v_{,r}^2\right)\right|
    \leq &\left( \int_0^1r\d r \int\d\varphi \left|2u_{r,r}+ v_{,r}^2\right|^2\right)^{1/2}\left(\int_h^1 \frac{\d r}{r}\right)^{1/2}\\
    \leq & C h|\log h|
  \end{split}
  \]
  from which we get
  \[
  \int_h^R |a(r)|\d r\leq C hR|\log h|\,.
  \]
  In conclusion,
  
  \begin{equation}
  \int_h^R |F(r)|\d r \leq ChR|\log h|\,.\label{eq:5}
  \end{equation}

  Using the standard interpolation inequality
\[
\|F'\|_{L^1(h,R)}\leq C \|F\|_{L^1(h,R)}^{1/2}\|F''\|_{L^1(h,R)}^{1/2}\,,
\]
and \eqref{eq:17}, we obtain the claim of the proposition.
\end{proof}

The proof of Theorem \ref{thm:main} now almost works exactly in the same way as
the proof of Theorem \ref{thm:main1}.

\begin{proof}[Proof of Theorem \ref{thm:main}]
The existence of a minimum follows from coercivity and lower
semi-continuity of the functional \eqref{eq:27}. The upper bound follows from
Lemma \ref{lem:upperbound2}. Assume that $(u,v)\in W^{1,2}(B_1;\R^3)\times
W^{2,2}(B_1;\R^3)$ with $\vk_{h,\Delta}\leq 2\pi\Delta^2h^2(|\log h|+C)$. By density of $C^2(B_1;\R^3)$ in $W^{1,2}(B_1;\R^3)\times
W^{2,2}(B_1;\R^3)$, we may assume $(u,v)\in C^2(B_1;\R^3)$. By Proposition
\ref{prop:interpolvk}, we may assume that the assumptions of Proposition
\ref{prop:diadicannuli} hold with $\alpha=1$ and
\[
y_1=y_2=0,\quad y_3=v\,.
\]
With this choice of $y_i$, we have $\kappa_y=\kvk_v$ and 
\[
\sum_{i=1}^3\int_{B_1}|D^2y_i|^2\d x=\int_{B_1}|D^2v|^2\d x\,.
\]
Hence, the lower bound follows from the statement of  Proposition
\ref{prop:diadicannuli}. This completes the proof of the theorem.
\end{proof}

\bibliographystyle{plain}
\bibliography{regular}

\end{document}